\documentclass[12]{amsart}
\usepackage{mathrsfs}
\usepackage{stmaryrd}
\usepackage[latin1]{inputenc}
\usepackage{amsmath}
\usepackage{amsfonts}
\usepackage{amssymb}
\usepackage[all,cmtip]{xy}
\usepackage{verbatim}   
\usepackage{amsthm}
\usepackage{amscd}
\usepackage{enumerate}
\usepackage{bbm} 
\usepackage{enumitem} 
\usepackage{color}
\definecolor{ghcolor}{RGB}{0, 150, 200} 
\definecolor{winestain}{rgb}{0.5,0,0}
\usepackage[linktocpage=true, hidelinks, colorlinks, linkcolor=blue, citecolor=red]{hyperref} 

\usepackage{url}

\theoremstyle{plain}
\newtheorem{thm}{Theorem}[section] 
\newtheorem{defn}[thm]{Definition}
\newtheorem{prop}[thm]{Proposition}

\newtheorem{lemma}[thm]{Lemma}
\newtheorem{sublemma}[thm]{Sublemma}
\newtheorem{corollary}[thm]{Corollary}

\theoremstyle{definition}
\newtheorem{rem}[thm]{Remark}

\newcommand{\huaSOE}{\mathfrak{S}_{\mathcal O_E}}

\newcommand{\hatL}{\hat{\mathfrak{L}}}
\newcommand{\hatm}{\hat{\mathfrak{M}}}
\newcommand{\hatM}{\hat{\mathfrak{M}}}

\newcommand{\barhatm}{\overline{\hat{\mathfrak{M}}}}

\newcommand{\barhatn}{\overline{\hat{\mathfrak{N}}}}

\newcommand{\barhatl}{\overline{\hat{\mathfrak{L}}}}

\newcommand{\barl}{\overline{\mathfrak{L}}}
\newcommand{\barm}{\overline{\mathfrak{M}}}
\newcommand{\barn}{\overline{\mathfrak{N}}}

\newcommand{\llb}{\llbracket}  
\newcommand{\rrb}{\rrbracket}

\newcommand{\vshape}{\varphi -\textnormal{shape}}
\newcommand{\vtshape}{(\varphi, \tau) -\textnormal{shape}}

\newcommand{\barrhosq}{\overline{\rho}^{\square}}
\newcommand{\rhosq}{\rho^{\square}}

\newcommand{\barhatmsq}{\barhatm^{\square}}

\newcommand{\barchi}{\overline{\chi}}
\newcommand{\barrho}{\overline{\rho}}
\newcommand{\rbar}{\overline{r}}

\newcommand{\barzeta}{\overline{\zeta}}

\newcommand{\Zp}{\mathbb{Z}_p}
\newcommand{\Qp}{\mathbb{Q}_p}
\newcommand{\Fp}{\mathbb{F}_p}

\newcommand{\Fpbar}{\overline{\mathbb{F}}_p}
\newcommand{\barFp}{\overline{\mathbb{F}}_p}

\newcommand{\barK}{\overline{K}}

\newcommand{\Z}{\mathbb{Z}}

\newcommand{\col}{\textnormal{col}}
\newcommand{\row}{\textnormal{row}}

\DeclareMathOperator{\Ext}{Ext}
\DeclareMathOperator{\Fil}{Fil}

\DeclareMathOperator{\Gal}{Gal}
\DeclareMathOperator{\GL}{GL}

\DeclareMathOperator{\Hom}{Hom}

\DeclareMathOperator{\Ker}{Ker}
\DeclareMathOperator{\Mat}{Mat}

\DeclareMathOperator{\Mod}{Mod}

\DeclareMathOperator{\Rep}{Rep}

\newcommand{\cris}{\mathrm{cris}}

\newcommand{\HT}{\mathrm{HT}}

\newcommand{\MF}{MF^{(\varphi, N)}}
\newcommand{\MFwa}{MF^{(\varphi, N)-\textnormal{w.a.}}}
\newcommand{\bigMF}{\mathcal{MF}^{(\varphi, N)}}
\newcommand{\bigMFwa}{\mathcal{MF}^{(\varphi, N)-\textnormal{w.a.}}}

\newcommand{\Acris}{A_{\textnormal{cris}}}

\newcommand{\D}{\mathcal{D}}
\newcommand{\huaS}{\mathfrak{S}}
\newcommand{\huaM}{\mathfrak{M}}

\newcommand{\huaN}{\mathfrak{N}}
\newcommand{\huaL}{\mathfrak{L}}
\newcommand{\hual}{\mathfrak{L}}
\newcommand{\huat}{\mathfrak{t}}

\newcommand{\Ghat}{\hat{G}}
\newcommand{\Rhat}{\hat{\mathcal{R}}}
\newcommand{\mhat}{\hat{\huaM}}
\newcommand{\That}{\hat{T}}

\newcommand{\tn}{t^{\{n\}}}

\newcommand{\bolde}{\boldsymbol{e}}

\newcommand{\boldr}{\boldsymbol{r}}

\newcommand{\boldt}{\boldsymbol{t}}
\newcommand{\boldf}{\boldsymbol{f}}

\title{Crystalline liftings and weight part of Serre's conjecture}
\author{HUI GAO}
\address{Beijing International Center for Mathematical Research, Peking University, No. 5 Yiheyuan Road, Haidian District, Beijing 100871, China}
\email{gaohui@math.pku.edu.cn}

\begin{document}

\begin{abstract}
We prove some new cases of weight part of Serre's conjectures for mod $p$ Galois representations associated to automorphic representations on unitary groups $U(d)$. The approach is a generalization of the work of Gee-Liu-Savitt, namely, we study reductions of certain crystalline representations, as well as crystalline lifts of these reductions.

\end{abstract}

\maketitle
\pagestyle{myheadings}
\markright{Crystalline liftings and weight part of Serre's conjecture}
\tableofcontents


\section*{Introduction}

Let $F$ be an imaginary CM field, and let $\rbar: G_F \to \GL_d(\Fpbar)$ be an irreducible representations. Suppose there exists a certain automorphic representation of $\GL_d(\mathbb A_F)$ whose associated Galois representation has reduction equal to $\rbar$. The (generalized) weight part of Serre's conjecture asks the following question: $\rbar$ is automorphic of what weights?

Beginning with the work of \cite{BDJ10}, there has been significant progress in establishing (generalized) weight part of Serre's conjectures. In particular, in the case $d=2$ and $p>2$, the problem (in the unitary group setting) is completely solved by \cite{GLS14, GLS13}, under some mild global hypothesis. Weight part of Serre's conjectures have also become increasingly important, in particular because of their role in formulating a $p$-adic Langlands correspondence (cf. \cite{BP12}).

The current paper aims to generalize the results of \cite{GLS14, GLS13} to higher dimensions (in the unramified case). The methods are similar to those in \textit{loc. cit.}, namely, we first study reductions of crystalline representations, and the liftings of these reductions. Then, we apply automorphy lifting theorems in \cite{BLGGT14} to conclude. To state our results more precisely, let us first introduce some notations.

Let $p>2$ be an odd prime, $K$ a finite unramified extension over $\Qp$ with residue field $k$, $\barK$ a fixed algebraic closure, and $G_{K}$ the absolute Galois group $\Gal(\barK/K)$.
Let $\mathcal S: =\{\kappa: K \hookrightarrow \barK\}$ be all the embeddings of $K$ into $\barK$. Fix one $\kappa_0 \in \mathcal S$, and recursively define $\kappa_{s+1} \in \mathcal S$ to be such that $\kappa_{s+1}^p \equiv \kappa_s (\bmod p)$. The subscripts are taken mod $f$, so $\kappa_f=\kappa_0$.
Fix a system of elements $\{\pi_n\}_{n=0}^{\infty}$ in $\barK$, where $\pi_0=\pi$ is a uniformizer of $K$, and $\pi_{n+1}^p=\pi_n, \forall n$. Let $K_n=K(\pi_n), K_{\infty}=\cup_{n=0}^{\infty}K(\pi_n)$, and $G_{\infty}:=\Gal(\barK/K_{\infty})$.

Let $E/\Qp$ be a finite extension that contains the image of every embedding of $K$ into $\barK$, $\mathcal O_E$ the ring of integers, $\omega_E$ a fixed uniformizer, $k_E=\mathcal O_E/\omega_E\mathcal O_E$ the residue field.
We can decompose the ring $K\otimes_{\Qp}E = \Pi_{s=0}^{f-1}E$, and let $1\in K\otimes_{\Qp}E$ map to $(\varepsilon_0, \ldots, \varepsilon_{f-1}) \in \Pi_{s=0}^{f-1}E$.

Let $V$ be a crystalline representation of $G_K$ over an $E$-vector space of dimension $d$.
Let $D$ be the associated filtered $\varphi$-module over $K\otimes_{\Qp} E$, which decomposes as $D=\Pi_{s=0}^{f-1} D_s$, where $D_s=\varepsilon_s D$.
Suppose $\HT_{\kappa_s}(D)=\HT(D_s)=\boldr_s = \{0 = r_{s, 1} < \ldots < r_{s, d} \leq p\}$ (note that we require $\min(\HT(D_s))=0$ for all $s$). Let $\rho=T$ be a $G_K$-stable $\mathcal O_E$-lattice in $V$, and $\overline \rho: =T/\omega_E T$ the reduction of $T$.

\begin{thm} \label{Thm1}
Suppose that the reduction $\overline \rho$ is upper triangular (i.e., successive extension of $d$ characters). Suppose that
\begin{itemize}
  \item Condition \textnormal{\textbf{(C-1)}} is satisfied, and
  \item Either \textnormal{\textbf{(C-2A)}} or \textnormal{\textbf{(C-2B)}} is satisfied.
\end{itemize}
Then there exists an upper triangular crystalline lift $\rho'$ of $\overline \rho$ such that $\HT_{s}(\rho)=\HT_{s}(\rho'), \forall s$.
\end{thm}

Here, Conditions \textnormal{\textbf{(C-1)}} and \textnormal{\textbf{(C-2A)}} (or \textnormal{\textbf{(C-2B)}}) are the technical conditions that we have to assume, see Section 3 and Section 6 respectively. We list several cases when these conditions are satisfied.

\begin{corollary} \label{introcor}
Suppose that the reduction $\overline \rho$ is upper triangular, i.e., there exists an increasing filtration $0= \Fil^ 0 \barrho \subset \Fil^1 \barrho \subset \ldots \subset \Fil^d \barrho =\barrho$ such that $\Fil^i \barrho/\Fil^{i-1}\barrho =\barchi_{d-i}, \forall 1 \leq i \leq d$, where $\barchi_i$ are some characters. If one of the following conditions is satisfied, then there exists an upper triangular crystalline lift $\rho'$ of $\overline \rho$ such that $\HT_{\kappa_s}(\rho)=\HT_{\kappa_s}(\rho')$ for all $0 \leq s \leq f-1$.
\begin{enumerate}
\item $K=\Qp$, the differences between two elements in $\HT(D_0)$ are never $p-1$. And $\overline \chi_i^{-1} \overline \chi_j \neq \mathbbm{1}, \overline \varepsilon_p, \forall i<j$, where $\mathbbm{1}$ is the trivial character, and $\overline \varepsilon_p$ is the reduction of the cyclotomic character.

  \item For each $s$, the differences between two elements in $\HT(D_s)$ are never $1$. And for one $s_0$, $p-1 \notin \HT(D_{s_0})$.And $\overline \chi_i^{-1} \overline \chi_j \neq \mathbbm{1}, \overline \varepsilon_p, \forall i<j.$

  \item For each $s$, the differences between two elements in $\HT(D_s)$ are never $1$. For one $s_0$, $p-1 \notin \HT(D_{s_0})$. For one $0 \leq s_0' \leq f-1$, $p \notin \HT(D_{s_0'})$ (it is possible that $s_0 =s_0'$).

   \item For each $s$, $\HT(D_s) \subseteq [0, p-1]$. And for one $s_0$, $p-1 \notin \HT(D_{s_0})$.
  \end{enumerate}
\end{corollary}

Theorem \ref{Thm1} is a (partial) generalization of the main local results in \cite{GLS14} to the higher dimensional case. However, our result is not complete as that in \cite{GLS14}, due to the technical conditions \textnormal{\textbf{(C-1)}} and \textnormal{\textbf{(C-2A)}} (or \textnormal{\textbf{(C-2B)}}). We will point out these difficulties in the paper.

Our local theorems have direct application to weight part of Serre's conjecture for higher dimensional representations (as outlined in \cite{BLGG14}). Here we give a sketchy statement of our theorem. See Theorem \ref{application} for full detail.

\begin{thm}
Suppose $p>2$. Let $F$ be an imaginary CM field, with maximal totally real subfield $F^{+}$ such that $F/F^{+}$ is unramified at all finite places, and all places $v$ in $F^{+}$ over $p$ splits completely in $F$. Furthermore, assume that $p$ is unramified in $F$.

Suppose $\overline r: G_F \to \GL_d(\barFp)$ is an irreducible representation such that $\overline r \simeq \overline r_{p, \iota}(\Pi)$, for an RACSDC automorphic representation $\Pi$ of $\GL_d(\mathbb{A}_F)$ with weights (whose corresponding Hodge-Tate weights are) in the Fontaine-Laffaille range and level prime to $p$. Suppose furthermore some usual Taylor-Wiles conditions are satisfied.

Now suppose that for \emph{each} $w\mid p$ in $F$, $\overline r \mid_{G_{F_w}}$ is upper triangular.
Let
$$a=(a_{w})_{w \mid p} \in (\mathbb Z_{+}^d)_0^{\coprod_{w \mid p}\Hom(k_w, \barFp)}$$
be a Serre weight, such that
\begin{itemize}
  \item $a_{w, \kappa, 1}- a_{w, \kappa, d} \leq p-d+1, \forall w\mid p, \kappa \in \Hom(k_w, \barFp)$, and
  \item $a_w \in W^{\textnormal{cris}}(\overline r \mid_{G_{F_w}}), \forall w\mid p.$
\end{itemize}

Suppose furthermore that for each $w\mid p$, any one of the listed 4 conditions in Theorem \ref{application} is satisfied. (These conditions directly correspond to the listed 4 conditions in Corollary \ref{introcor}).

Then, $\overline r$ is automorphic of weight $a$.

\end{thm}

\noindent \textbf{Remarks on some related papers.} Our paper gives the first ``general" evidence towards the weight part of Serre's conjectures, with no restriction on the dimension $d$ or the unramified base field $K$, and \emph{outside} the Fontaine-Laffaille range (although of course with many restriction on the Hodge-Tate weights). Here are some remarks on some related papers.

The case when $d=2$ is fully solved (when $p>2$) by \cite{GLS14} (when $K$ is unramified) and \cite{GLS13} (when $K$ is ramified). Our paper is a direct generalization of \cite{GLS14, GLS13}. In \textit{loc. cit.}, many results are proved in an \textit{ad hoc} way since $d=2$. The key insight in our paper is that we can reprove several theorems in \textit{loc. cit.} in a way that can be generalized to higher dimensions (e.g., Proposition \ref{shape} in our paper). Also, with the help and inspiration from the unpublished notes \cite{GLS14+}, we are able to formulate our crystalline lifting theorems in a new way (see Section 5 and Section 7) that can lead to better understanding of $(\varphi, \Ghat)$-modules. In particular, we can generalize Theorem \ref{Thm1} to the case where $K$ is ramified (see forthcoming \cite{Gao15ram}). We also note that by using some different crystalline lifting techniques (inspired by the work of \cite{BH15}), we can strengthen the result in Corollary \ref{introcor}(1) (i.e., the $K=\Qp$ case) in the forthcoming \cite{Gao15Qp}.

All the results mentioned in the previous paragraph have the limitation that we can only treat those Serre weights corresponding to Hodge-Tate weights in the range $[0, p]$. This is a serious limitation when the dimension $d>2$. The first breakthrough was due to \cite{EGH13} in the case when $K=\Qp$, $d=3$ and $\barrho$ is absolutely irreducible, using a technique called weight cycling. Also, the paper \cite{BLGG14} treated the case when $K=\Qp$ and $\barrho$ is semisimple of any dimension $d$ (see Corollary 4.1.14 of \textit{loc. cit.}). In particular, when $d=3$, there is a much more explicit and detailed result (Theorem 5.1.4 of \textit{loc. cit.}), although the detailed calculations for $d=3$ seem quite difficult to generalize to higher dimensions. A more recent substantial advance (after our paper was posted) is due to \cite{LLHLM}, where they completely proved the Serre weight conjectures (as conjectured by \cite{Her09}) in the case $K=\Qp$, $d=3$ and $\barrho$ is semi-simple (and generic). The key insight of \cite{LLHLM} is that they can explicitly compute certain \emph{potentially} crystalline deformation rings (whereas in the current paper, we only consider crystalline representations).

\smallskip

\noindent \textbf{Strategy for the main local results.}
Now, let us sketch the strategy of the proof of our main local result (Theorem \ref{Thm1}). The strategy follows closely that of \cite{GLS14} and \cite{GLS13}, and uses an induction process.

Let $\huaM$ be the Kisin module attached to $T$, which is a free module of rank $d$ over $W(k)[\![u]\!]\otimes_{\Zp} \mathcal O_E$, and decomposes as $\huaM =\Pi_{s=0}^{f-1} \huaM_s$, where $\huaM_s=\varepsilon_s\huaM$. We regard $\huaM_{s-1}$ as a $\varphi(\mathcal O_E[\![u]\!])$-submodule of $\huaM_s^{\ast}$. The first step in our proof is to control the shape the reduction of the Kisin module (i.e., the shape of $\barm$) associated to certain crystalline representations, and we can give an \emph{upper bound} for the number of these shapes.

Kisin modules only give us information on the $G_{\infty}$-action on representations. However, when the reduction of $T$ is upper triangular, under certain technical conditions, we can show that the $G_{\infty}$-action already determines the $G_K$-action. All these results put restrictions on the possible shapes of reductions of crystalline representations when the reduction is upper triangular. The number of possible shapes of reductions is also bounded by the upper bound we mentioned in the last paragraph.

The next step is to show that upper triangular crystalline representations already give rise to enough upper triangular reductions. We can give a precise number of all these upper triangular reductions coming this way, which happens to be exactly the same as the upper bound mentioned in the last paragraph. Thus by pigeonhole principle, all upper triangular reductions of crystalline representations can be obtained by reductions of upper triangular crystalline representations, and Theorem \ref{Thm1} is proved.

\smallskip
\noindent \textbf{Structure of the paper.}
We now explain the structure of this paper. In Section 1, we review the theory of Kisin modules and $(\varphi, \Ghat)$-modules with $\mathcal O_E$-coefficient (and $k_E$-coefficient). In particular, we review the structure of rank-1 modules. In Section 2, we take the first step in studying the shape of upper triangular Kisin modules with $k_E$-coefficient (which come from reduction of crystalline representations). In Section 3, we introduce the Condition \textbf{(C-1)}, which helps to avoid certain complication in the shape of upper triangular Kisin modules with $k_E$-coefficient. Then in Section 4, with the assumption \textbf{(C-1)}, we can continue the studies in Section 2, and give an \emph{upper bound} for the shapes of upper triangular Kisin modules with $k_E$-coefficient that we study. In Section 5, we show that certain set of extension classes have natural $\mathcal O_E$-module (and sometimes, $k_E$-vector space) structures, which will be used in the induction process. We also show that these $\mathcal O_E$-module structures are compatible with each other. In Section 6, we introduce the two conditions \textbf{(C-2A)} and \textbf{(C-2B)}. When either of the two conditions is satisfied, then roughly speaking, the $G_{\infty}$-information that Kisin modules carry actually determine the full $G_K$-information. In Section 7, we prove our main local result, combining everything in the previous sections. It relies on an induction process, where the $d=2$ case is proved in \cite{GLS14}. Finally in Section 8, we apply our local results to weight part of Serre's conjecture.


\smallskip
\noindent\textbf{Notations.}
In this paper, we frequently use boldface letters (e.g., $\bolde$) to mean a sequence of objects (e.g., $\bolde=(e_1, \ldots, e_d)$ a basis of some module). We use $\Mat(?)$ to mean the set of matrices with elements in $?$. We use notations like $[u^{r_1}, \ldots, u^{r_d}]$ to mean a diagonal matrix with the diagonal elements in the bracket. We use $Id$ to mean the identity matrix.


In this paper, \textbf{upper triangular} always means successive extension of rank-$1$ objects.
We use notations like $\mathcal E(m_d, \ldots, m_1)$ (note the order of objects) to mean the set of all upper triangular extensions of rank-1 objects in certain categories. That is, $m$ is in $\mathcal E(m_d, \ldots, m_1)$ if there is an increasing filtration $0=\Fil^0 m \subset \Fil^1 m \subset \ldots \subset \Fil^d m =m$ such that $\Fil^i m /\Fil^{i-1}m =m_i, \forall 1 \leq i \leq d$. \emph{Note} that we do not define any equivalence relations between elements in this set. This is in contrast with some other sets which we define in Section 6 (with notations $\Ext(\ast, \ast)$).

We normalize the Hodge-Tate weights so that $\HT_{\kappa}(\varepsilon_p)={1}$ for any $\kappa: K \to \overline{K}$, where $\varepsilon_p$ is the $p$-adic cyclotomic character.

We recall some notations in $p$-adic Hodge theory and integral $p$-adic Hodge theory. All these notations in fact work for all $K/\Qp$ with no ramification restriction.

We fix a system of elements $\{\mu_{p^n}\}_{n=0}^{\infty}$ in $\barK$, where $\mu_1=1$, $\mu_p$ is a primitive $p$-th root of unity, and $\mu_{p^{n+1}}^p=\mu_{p^n}, \forall n$.
Let $K_{p^{\infty}} = \cup_{n=0}^{\infty} K(\mu_{p^n})$, and $\hat{K}=K_{\infty, p^{\infty}} = \cup_{n=0}^{\infty} K(\pi_n, \mu_{p^n}).$
Note that $\hat{K}$ is the Galois closure of $K_{\infty}$, and let
$\Ghat =\Gal(\hat{K}/K)$, $H_K= \Gal(\hat{K}/K_{\infty})$, and $ G_{p^{\infty}} = \Gal(\hat{K}/K_{p^{\infty}}).$
When $p>2$, then $\hat G \simeq G_{p^{\infty}} \rtimes H_K$ and $G_{p^{\infty}} \simeq \Zp(1)$ by \cite[Lem. 5.1.2]{Liu08}, and so we can (and do) fix a topological generator $\tau$ of $G_{p^{\infty}}$. And we can furthermore assume that $\mu_{p^n}=\frac{\tau(\pi_n)}{\pi_n}$ for all $n$.

Let $C=\hat{\barK}$ be the completion of $\barK$, with ring of integers $\mathcal O_C$. Let $R: = \varprojlim \mathcal O_C/p$ where the transition maps are $p$-th power map. $R$ is a valuation ring with residue field $\bar k$ ($\bar k$ is the residue field of $C$). $R$ is a perfect ring of characteristic $p$. Let $W(R)$ be the ring of Witt vectors. Let $\underline \epsilon :=(\mu_{p^n})_{n=0}^{\infty} \in R$, $\underline \pi =(\pi_n)_{n=0}^{\infty} \in R$, and let $[\underline \epsilon], [\underline \pi]$ be their Teichm\"{u}ller representatives respectively in $W(R)$.

There is a map $\theta: W(R) \to \mathcal O_C$ which is the unique universal lift of the map $R \to \mathcal O_C/p$ (projection of $R$ onto the its first factor), and $\Ker \theta$ is a principle ideal generated by $\xi= [\overline \omega]+p$, where $\overline{\omega} \in R$ with $\omega^{(0)}=-p$, and $[\overline \omega] \in W(R)$ its Teichm\"{u}ller representative. Let $B_{\rm{dR}}^{+} := \varprojlim_{n} W(R)[\frac{1}{p}]/(\xi)^n$, and $B_{\rm{dR}}:=B_{\rm{dR}}^{+}[\frac{1}{\xi}]$. Let $t: = \log([\underline{\epsilon}])$, which is an element in $B_{\rm{dR}}^{+}$.

Let $A_{\rm cris}$ denote the $p$-adic completion of the divided power envelope of $W(R)$ with respect to $\Ker(\theta)$. Let $B_{\rm cris}^{+} = A_{\rm cris}[1/p]$ and $B_{\rm cris}:= B_{\rm cris}^{+}[\frac{1}{t}]$. Let $B_{\rm st}:=B_{\rm cris}[X]$ where $X$ is an indeterminate. There are natural Frobenius actions, monodromy actions and filtration structures on $B_{\rm cris}$ and $B_{\rm st}$, which we omit the definition. We have the natural embeddings $B_{\rm cris} \subset B_{\rm st} \subset B_{\rm{dR}}$.

Let $\huaS: = W(k)\llb u\rrb$, $E(u)\in W(k)[u]$ the minimal polynomial of $\pi$ over $W(k)$, and $S$ the $p$-adic completion of the PD-envelope of $\huaS$ with respect to the ideal $(E(u))$.
We can embed the $W(k)$-algebra $W(k)[u]$ into $W(R)$ by mapping $u$ to $[\underline \pi]$. The embedding extends to the embeddings $\huaS \hookrightarrow S \hookrightarrow A_{\rm cris}$.

The projection from $R$ to $\overline{k}$ induces a projection $\nu : W(R) \to W(\overline{k})$, since $\nu(\Ker \theta) = pW(\overline{k})$, the projection extends to $\nu: \Acris \to W(\overline{k})$, and also $\nu: B_{\text{cris}}^{+} \to W(\overline{k})[\frac{1}{p}]$. Write
$I_{+}B_{\text{cris}}^{+}:=\Ker(\nu: B_{\text{cris}}^{+} \to W(\overline{k})[\frac{1}{p}]),$
and for any subring $A \subseteq B_{\text{cris}}^{+}$, write $I_{+}A = A\cap \Ker(\nu)$. Also, define
$I^{[n]}B_{\text{cris}}^{+} := \{ x \in  B_{\text{cris}}^{+}: \varphi^k(x) \in \Fil^n B_{\text{cris}}^{+}, \text{ for all } k>0  \},$
and for any subring $A \subseteq B_{\text{cris}}^{+}$, write $I^{[n]}A:= A\cap I^{[n]}B_{\text{cris}}^{+}$.

There exists a nonzero $\huat \in W(R)$ such that $\varphi(\huat) = c_0^{-1}E(u) \huat$, where $c_0 p$ is the constant term of $E(u)$. Such $\huat$ is unique up to units in $\Zp$, and we can select a such $\huat$ such that $t= \lambda \varphi(\huat)$ with $\lambda = \Pi_{n=1}^{\infty} \varphi^n(\frac{c_0^{-1}E(u)}{p}) \in S^{\times}$. For all $n$, $I^{[n]}W(R)$ is a principle ideal, and by \cite[Lem. 3.2.2]{Liu10}, $(\varphi(\huat))^n$ is a generator of the ideal.

\smallskip
\noindent \textbf{Acknowledgement} This paper is a natural generalization of the results of Toby Gee, Tong Liu and David Savitt. It is a great pleasure to acknowledge their beautiful papers. In particular, the author would like to thank their great generosity in sharing with us their unpublished notes \cite{GLS14+}, which played an very important role in the later developments of this paper (see Section 5 of our paper).
We would like to thank Tong Liu and David Savitt for comments on an earlier draft.
This paper is written when the author is a postdoc in Beijing International Center for Mathematical Research, and we would like to thank the institute for the hospitality. The author also would like to heartily thank his postdoc mentor, Ruochuan Liu, for his constant interest, encouragement and support.
We would like to thank Florian Herzig, Stefano Morra, Yoshiyasu Ozeki, Chol Park, Zhongwei Yang, and Yuancao Zhang for various useful discussions and correspondences. The author also thank the anonymous referee(s) for useful comments which help to improve the exposition.
This work is partially supported by China Postdoctoral Science Foundation General Financial Grant 2014M550539.


\section{Kisin modules and $(\varphi, \Ghat)$-modules with coefficients}

\subsection{Kisin modules and $(\varphi, \Ghat)$-modules with coefficients}
In this subsection, we recall useful facts in the theory of Kisin modules and $(\varphi, \Ghat)$-modules with $\mathcal O_E$-coefficients. All results in this subsection work for any $K/\Qp$ and any $p$.

Recall that $\mathfrak{S}=W(k)[\![u]\!]$ with the Frobenius endomorphism $\varphi_{\huaS}: \huaS \to \huaS$ which acts on $W(k)$ via arithmetic Frobenius and sends $u$ to $u^p$.
Denote
$\huaS_{\mathcal O_E}:= \huaS \otimes_{\Zp}\mathcal O_E$ and $\huaS_{k_E}:= \huaS \otimes_{\Zp}k_E = k[\![u]\!] \otimes_{\Fp} k_E$,
and extend $\varphi_{\huaS}$ to $\huaS_{\mathcal O_E}$ (resp. $\huaS_{k_E}$) by acting on $\mathcal O_E$ (resp. $k_E$) trivially. Let $r$ be any nonnegative integer.

\begin{itemize}
 \item  Let $'\Mod_{\huaS_{\mathcal O_E}}^{\varphi}$ (called the category of Kisin modules of height $r$ with $\mathcal O_E$-coefficients) be the category whose objects are $\huaS_{\mathcal O_E}$-modules $\huaM$, equipped with $\varphi:\huaM\to\huaM$ which is a
$\varphi_{\huaS_{\mathcal O_E}}$-semi-linear morphism such that the span of $\text{Im}(\varphi)$ contains $E(u)^{r}\huaM$.
The morphisms in the category are $\huaS_{\mathcal O_E}$-linear maps that commute with $\varphi$.

\item Let $\Mod_{\huaS_{\mathcal O_E}}^{\varphi}$ be the full subcategory of $'\Mod_{\huaS_{\mathcal O_E}}^{\varphi}$ with $\huaM \simeq \oplus_{i \in I}\huaS_{\mathcal O_E}$ where $I$ is a finite set. Let $\Mod_{\huaS_{k_E}}^{\varphi}$ be the full subcategory of $'\Mod_{\huaS_{\mathcal O_E}}^{\varphi}$ with $\huaM \simeq \oplus_{i \in I}\huaS_{k_E}$ where $I$ is a finite set.
\end{itemize}

For any integer $n \geq 0$, write $n =(p-1)q(n)+r(n)$ with $q(n)$ and $r(n)$ the quotient and residue of $n$ divided by $p-1$. Let $\tn=(p^{q(n)}\cdot q(n)!)^{-1}\cdot t^n$, we have $\tn \in A_{\text{cris}}$.
We define a subring of $B_{\text{cris}}^{+}$,
$\mathcal{R}_{K_0} :=\left\{ \sum_{i=0}^{\infty} f_i t^{\{i\}} , f_i \in S_{K_0}, f_i \to 0 \text{ as } i \to \infty \right\}.$
Define $\hat{\mathcal{R}}:= \mathcal{R}_{K_0} \cap W(R)$. Then $\Rhat$ is a $\varphi$-stable subring of $W(R)$, which is also $G_K$-stable, and the $G_K$-action factors through $\Ghat$. Denote $\Rhat_{\mathcal O_E}:= \Rhat \otimes_{\Zp}\mathcal O_E, \quad W(R)_{\mathcal O_E}:= W(R) \otimes_{\Zp}\mathcal O_E,$
and extend the $G_K$-action and $\varphi$-action on them by acting on $\mathcal O_E$ trivially.
Note that $\huaS_{\mathcal O_E} \subset \hat R_{\mathcal O_E}$, and let $\varphi: \huaS_{\mathcal O_E} \to \hat R_{\mathcal O_E}$ be the composite of $\varphi_{\huaS_{\mathcal O_E}}: \huaS_{\mathcal O_E} \to \huaS_{\mathcal O_E}$ and the embedding $\huaS_{\mathcal O_E} \to \hat R_{\mathcal O_E}$.

\begin{defn}
Let $'\Mod_{\huaS_{\mathcal O_E}}^{\varphi, \Ghat}$ be the category (called the category of $(\varphi, \Ghat)$-modules of height $r$ with $\mathcal O_E$-coefficients) consisting of triples $(\huaM, \varphi_{\huaM}, \Ghat)$ where,
\begin{enumerate}
\item $(\huaM, \varphi_{\huaM}) \in '\Mod_{\huaS_{\mathcal O_E}}^{\varphi}$ is a Kisin module of height $r$;
\item $\Ghat$ is a $\Rhat_{\mathcal O_E}$-semi-linear $\Ghat$-action on $\mhat := \Rhat_{\mathcal O_E} \otimes_{\varphi, \huaS_{\mathcal O_E}} \huaM$;
\item $\Ghat$ commutes with $\varphi_{\mhat} : =\varphi_{\Rhat_{\mathcal O_E}}\otimes \varphi_{\huaM}$;
\item Regarding $\huaM$ as a $\varphi(\huaS_{\mathcal O_E})$-submodule of $\mhat$, then $\huaM \subseteq \mhat^{H_K}$;
\item $\Ghat$ acts on the $\mhat/(I_{+}\hat{R})\mhat$ trivially.
\end{enumerate}
A morphism between two $(\varphi, \Ghat)$-modules is a morphism in $\Mod_{\huaS_{\mathcal O_E}}^{\varphi}$ which commutes with $\Ghat$-actions.
\end{defn}

Let $\Mod_{\huaS_{\mathcal O_E}}^{\varphi, \Ghat}$ be the full subcategory of $'\Mod_{\huaS_{\mathcal O_E}}^{\varphi, \Ghat}$ where $\huaM \in \Mod_{\huaS_{\mathcal O_E}}^{\varphi}$.
Let $\Mod_{\huaS_{k_E}}^{\varphi, \Ghat}$ be the full subcategory of $'\Mod_{\huaS_{\mathcal O_E}}^{\varphi, \Ghat}$ where $\huaM \in \Mod_{\huaS_{k_E}}^{\varphi}$.

We summarize some useful results about Kisin modules and $(\varphi, \Ghat)$-modules with coefficients.
\begin{thm} We can associate representations to Kisin modules and $(\varphi, \Ghat)$-modules.
  \begin{enumerate}
    \item Suppose $\huaM \in \Mod_{\huaS_{\mathcal O_E}}^{\varphi}$ of $\huaS_{\mathcal O_E}$-rank $d$, then
        \begin{itemize}
          \item  $T_{\huaS}(\huaM): = \Hom_{\huaS, \varphi}(\huaM, W(R))$ and
          \item $T_{\huaS_{\mathcal O_E}}(\huaM): = \Hom_{\huaS_{\mathcal O_E}, \varphi}(\huaM, W(R)_{\mathcal O_E})$
        \end{itemize}
    are naturally isomorphic as finite free $\mathcal O_E$-representations of $G_{\infty}$ of rank $d$.

\item Suppose ${\huaM} \in \Mod_{\huaS_{k_E}}^{\varphi}$ of $\huaS_{k_E}$-rank $d$, then
\begin{itemize}
  \item $T_{\huaS}({\huaM}): = \Hom_{\huaS, \varphi}(\huaM, W(R)\otimes_{\Zp}\Qp/\Zp)$ and
  \item $T_{\huaS_{k_E}}({\huaM}): = \Hom_{\huaS_{k_E}, \varphi}(\huaM, W(R)_{\mathcal O_E}\otimes_{\Zp}\Qp/\Zp)$
\end{itemize}
are naturally isomorphic as $k_E$-representations of $G_{\infty}$ of dimension $d$.

\item Suppose $\mhat \in \Mod_{\huaS_{\mathcal O_E}}^{\varphi, \Ghat}$ where $\huaM$ is of $\huaS_{\mathcal O_E}$-rank $d$, then
\begin{itemize}
  \item $ \That(\mhat) := \Hom_{\Rhat, \varphi} (\mhat, W(R))$ and
  \item $ \That_{\huaS_{\mathcal O_E}}(\mhat) := \Hom_{\Rhat_{\mathcal O_E}, \varphi} (\mhat , W(R)_{\mathcal O_E})$
\end{itemize}
are naturally isomorphic as finite free $\mathcal O_E$-representations of $G_K$ of rank $d$.

\item Suppose $\mhat \in \Mod_{\huaS_{k_E}}^{\varphi, \Ghat}$ where $\huaM$ is of $\huaS_{k_E}$-rank $d$, then
\begin{itemize}
  \item $ \That(\mhat) := \Hom_{\Rhat, \varphi} (\mhat, W(R)\otimes_{\Zp}\Qp/\Zp)$ and
  \item $ \That_{\huaS_{k_E}}(\mhat) := \Hom_{\Rhat_{\mathcal O_E}, \varphi} (\mhat , W(R)_{\mathcal O_E}\otimes_{\Zp}\Qp/\Zp)$
\end{itemize}
are naturally isomorphic as $k_E$-representations of $G_K$ of dimension $d$.

  \end{enumerate}
\end{thm}
\begin{proof}
These statements are first prove in \cite{Kis06, Liu10} without considering the $\mathcal O_E$ (or $k_E$)-coefficients. For the proof concerning $\mathcal O_E$ (or $k_E$)-coefficients, and the isomorphisms between the two ways of constructing representations, see \cite[Prop. 3.4, Thm. 5.2]{GLS14}, as well as \cite[Prop. 9.1.8]{Lev13}.
\end{proof}

\begin{thm} \label{huaMcoeff} \hfill
 \begin{enumerate}
  \item For $\huaM \in \Mod_{\huaS_{\mathcal O_E}}^{\varphi}$, we have $T_{\huaS}(\huaM/\omega_E\huaM)\simeq T_{\huaS}(\huaM)/\omega_ET_{\huaS}(\huaM).$
  \item The functor $T_{\huaS}: \Mod_{\huaS_{\mathcal O_E}}^{\varphi} \to \Rep_{\mathcal O_E}(G_{\infty})$ is exact and fully faithful.
  \item Suppose $V$ is a semi-stable representation of $G_K$ over an $E$-vector space, with Hodge-Tate weights in $\{0, \ldots, r\}$ when considering $V$ as a $\Qp$-vector space. Suppose $L \subset V$ is a $G_{\infty}$-stable $\mathcal O_E$-lattice, then there exists $\huaM \in \Mod_{\huaS_{\mathcal O_E}}^{\varphi}$, such that $T_{\huaS}(\huaM) \simeq L$.

\end{enumerate}
\end{thm}

\begin{thm} \label{hatMcoeff} \hfill
 \begin{enumerate}
  \item For $\mhat \in \Mod_{\huaS_{\mathcal O_E}}^{\varphi, \Ghat}$, we have $\That(\mhat/\omega_E\mhat) \simeq \That(\mhat)/\omega_E\That(\mhat)$.
\item There is natural isomorphism $\That(\mhat) \mid_{G_{\infty}} \simeq T_{\huaS}(\huaM)$ as $\mathcal O_E[G_{\infty}]$-representations.
  \item $\That$ induces an anti-equivalence between the category $\Mod_{\huaS_{\mathcal O_E}}^{\varphi, \Ghat}$ and the category of $G_K$-stable $\mathcal O_E$-lattices in semi-stable $E$-representations of $G_K$ with Hodge-Tate weights in $\{0, \ldots, r\}$ (when considering $V$ as a $\Qp$-vector space).
\end{enumerate}
\end{thm}

\begin{rem}
The proof of the statements in Theorem \ref{huaMcoeff} and Theorem \ref{hatMcoeff} can either be found or easily deduced, from \cite[\S 3, \S 5.1]{GLS14} and \cite[\S 4]{Lev14}. They are in turn, based on works in \cite{Kis06, Liu10, CL11}, where they developed the theory without
$\mathcal O_E$-coefficients. We also remark that Statement (3) in Theorem \ref{hatMcoeff} first appeared in \cite[Thm. 4.1.6]{Lev14}.
\end{rem}

\subsection{$(\varphi, \Ghat)$-modules when $p>2$}
When $p>2$, the theory of $(\varphi, \Ghat)$-modules becomes simpler.

\begin{lemma} \label{pnot2}
Suppose $p>2$. Let $\mhat \in \Mod_{\huaS_{\mathcal O_E}}^{\varphi, \Ghat}$. Then $\mhat$ is uniquely determined up to isomorphism by the following information:
\begin{enumerate}
  \item A matrix $A_{\varphi} \in \Mat(\huaS_{\mathcal O_E})$ for the Frobenius $\varphi: \huaM \to \huaM$, such that there exist $B \in \Mat(\huaS_{\mathcal O_E})$ with $A_{\varphi}B=E(u)^r Id$.
   \item A matrix $A_{\tau} \in \Mat(\hat{R}_{\mathcal O_E})$ (for the $\tau$-action $\tau: \mhat \to \mhat$) such that
   \begin{itemize}
     \item $A_{\tau}-Id \in \Mat(I_{+}\Rhat_{\mathcal O_E}),$
     \item $A_{\tau}\tau(\varphi(A_{\varphi}))=\varphi(A_{\varphi})\varphi(A_{\tau}).$
     \item $g(A_{\tau}) = \prod_{k=0}^{\varepsilon_p(g)-1} \tau^k(A_{\tau})$ for all $g\in G_{\infty}$ such that $\varepsilon_p(g) \in \mathbb{Z}^{\geq 0}$.
   \end{itemize}

\end{enumerate}
\end{lemma}
\begin{proof}
This is because when $p>2$, we have $\hat G \simeq G_{p^{\infty}} \rtimes H_K$, and $G_{p^{\infty}}$ is topologically generated by $\tau$. The last bullet item ($g(A_{\tau}) = \prod_{k=0}^{\varepsilon_p(g)-1} \tau^k(A_{\tau})$) in Condition (2) is needed by \cite[Prop. 1.3]{Car13}.
\end{proof}

We can detect $(\varphi, \Ghat)$-modules that are crystalline, by the following theorem.

\begin{thm}[{\cite[Prop. 5.9]{GLS14}, \cite[Thm. 21]{Oze14}}] \label{Ozeki}
Suppose $p>2$, and let $\mhat \in \Mod_{\huaS_{\mathcal O_E}}^{\varphi, \Ghat}$. Then $\hat T(\mhat)\otimes_{\mathcal O_E}E$ is a crystalline representation if and only if
$$(\tau-1 )(\huaM) \in \mhat \cap (u^p\varphi(\mathfrak t) W(R)\otimes_{\varphi, \huaS}\huaM) =\mhat \cap (u^p\varphi(\mathfrak t) W(R)_{\mathcal O_E}\otimes_{\varphi, \huaS_{\mathcal O_E}}\huaM).$$
\end{thm}

\begin{rem}\label{remarkOzeki}
By Theorem \ref{Ozeki} and Lemma \ref{pnot2}, when $p>2$, to give a crystalline $(\varphi, \Ghat)$-module $\mhat \in \Mod_{\huaS_{\mathcal O_E}}^{\varphi, \Ghat}$ is the same to give
\begin{enumerate}
\item  A matrix $A_{\varphi} \in \Mat(\huaS_{\mathcal O_E})$ as in Lemma \ref{pnot2}(1).
\item A matrix $A_{\tau} \in \Mat(\hat{R}_{\mathcal O_E})$ as in Lemma \ref{pnot2}(2), except that we furthermore require $A_{\tau}-Id \in \Mat(\Rhat_{\mathcal O_E} \cap (u^p\varphi(\mathfrak t)W(R)_{\mathcal O_E}))$.

\end{enumerate}
\end{rem}

\subsection{Rank 1 Kisin modules and $(\varphi, \Ghat)$-modules}

Now, we recall some useful facts about rank-$1$ Kisin modules and $(\varphi, \Ghat)$-modules with $\mathcal O_E$-coefficients. See \cite[\S 6]{GLS14} and \cite[\S 5.1]{GLS13} for more details. In this subsection, we have to assume that $K/\Qp$ is unramified and $p$ any prime number (except Lemma \ref{solution}, where we assume $p>2$).

\begin{defn}
Let $\boldt = (t_0, \ldots, t_{f-1})$ be a sequence of non-negative integers, $a \in k_{E}^{\times}$. Let $\barm(\boldt; a):=\barm(t_0, \ldots, t_{f-1}; a)= \prod_{s=0}^{f-1}\barm(\boldt; a)_s$ be the rank-$1$ module in $\Mod_{\huaS_{k_E}}^{\varphi}$ such that
\begin{itemize}
\item $\barm(\boldt; a)_s$ is generated by $e_s$, and
\item $\varphi(e_{s-1})=(a)_s u^{t_s}e_s$, where $(a)_s=a$ if $s=0$ and $(a)_s=1$ otherwise.
\end{itemize}
\end{defn}

\begin{defn} \label{huaMrank1}
Let $\boldt = (t_0, \ldots, t_{f-1})$ be a sequence of non-negative integers, $\hat a \in \mathcal O_E$. Let $\huaM(\boldt; \hat a):=\huaM(t_0, \ldots, t_{f-1}; \hat a)= \prod_{s=0}^{f-1}\huaM(\boldt; \hat a)_s$ be the rank-$1$ module in $\Mod_{\huaS_{\mathcal O_E}}^{\varphi}$ such that
\begin{itemize}
\item $\huaM(\boldt; \hat a)_s$ is generated by $\tilde e_s$, and
\item $\varphi(\tilde e_{s-1})=(\hat a)_s (u-\pi)^{t_s}\tilde e_s$, where $(\hat a)_s=\hat a$ if $s=0$ and $(\hat a)_s=1$ otherwise.
\end{itemize}

\end{defn}

\begin{lemma}[{\cite[Lem. 6.2, Lem. 6.3, Cor. 6.5]{GLS14}}]\label{rank1}
\hfill
  \begin{enumerate}
    \item Any rank $1$ module in $\Mod_{\huaS_{k_E}}^{\varphi}$ is of the form $\barm(\boldt; a)$ for some $\boldt$ and $a$.
    \item When $\hat a$ is a lift of $a$, $\huaM(\boldt; \hat a)/\omega_E\huaM(\boldt; \hat a) \simeq \barm(\boldt; a)$.
    \item There is a unique $\hat \huaM(\boldt; \hat a) \in \Mod_{\huaS_{\mathcal O_E}}^{\varphi, \Ghat}$ such that
    \begin{itemize}
      \item The ambient Kisin module of $\hat \huaM(\boldt; \hat a)$ is $\huaM(\boldt; \hat a)$, and
      \item $\That(\hat \huaM(\boldt; \hat a))$ is a crystalline character.
    \end{itemize}
    And in fact,
    $\That(\hat \huaM(\boldt; \hat a))= \lambda_{\hat a}\prod_{s=0}^{f-1}\psi_s^{t_s},$
    where $\psi_s$ is certain crystalline character such that $\HT_i(\psi_s)=1$ if $i=s$, $\HT_i(\psi_s)=0$ if $i \neq s$, and $\lambda_{\hat a}$ is the unramified character of $G_K$ which sends the arithmetic Frobenius to $\hat a$.
    \item There is a unique $\overline{\hatm}(\boldt; a) \in \Mod_{\huaS_{k_E}}^{\varphi, \Ghat}$ such that the ambient Kisin module is
    $\barm(\boldt; a)$. Furthermore, $\That(\overline{\hatM}(\boldt; a))$ is the reduction of $\That(\hat \huaM(\boldt; \hat a))$ for any lift $\hat a \in \mathcal O_E$ of $a$.
  \end{enumerate}
\end{lemma}

\begin{defn}
Let $\barn =\barm(\boldt ;a)$, for each $s$, define $\alpha_s(\barn): = \frac{1}{p^f-1} \sum_{j=1}^f p^{f-j}t_{j+s}$. Note that we have $\alpha_s(\barn)+t_s =p\alpha_{s-1}(\barn), \forall s$.
\end{defn}

\begin{lemma}[{\cite[Lem. 5.1.2]{GLS13}, \cite[Prop. 6.7]{GLS14}}] \label{comparerank1}
Let $\barhatn=\barhatm(\boldt; a)$, $\barhatn'=\barhatm(\boldt'; a')$, then
  \begin{enumerate}
  \item $\That(\barhatn) \mid_{I_K} \simeq \Pi_{s=0}^{f-1} w_s^{t_s}$, where $w_s: I_K \to \overline{\Fp}^{\times}$ is the fundamental character corresponding to $\overline \kappa_s : k \hookrightarrow \overline{\Fp}$ (the reduction of $\kappa_s$).
  \item The following are equivalent:
  \begin{enumerate}
    \item $\That(\barhatn)\simeq \That(\barhatn')$ as $G_K$-representations.
    \item $T_{\huaS}(\barn) \simeq T_{\huaS}(\barn')$ as $G_{\infty}$-representations.
    \item $\alpha_s(\barn)-\alpha_s(\barn') \in \mathbb Z$ for some $s$ (and thus all $s$), and $a=a'$.
    \item $\sum_{s=0}^{f-1}p^{f-1-s}t_s \equiv \sum_{s=0}^{f-1}p^{f-1-s}t_s' (\bmod p^f-1)$, and $a=a'$.
  \end{enumerate}
    \item There exists nonzero morphism $\barn \to \barn'$ if and only if $\alpha_s(\barn)-\alpha_s(\barn') \in \mathbb Z_{\geq 0}$ for all $s$, and $a=a'$.
     \end{enumerate}
\end{lemma}

We recall the following useful lemma.
\begin{lemma}[{\cite[Lem. 7.1]{GLS14}}] \label{solution}
Let $p>2$, $t_0, \ldots, t_{f-1} \in [-p, p]$ such that $\sum_{s=0}^{f-1}p^{f-1-s}t_s \equiv 0 (\bmod p^f-1)$. Then one of the following holds:
\begin{enumerate}
 \item $(t_0, \ldots, t_{f-1})= \pm (p-1, \ldots, p-1)$,
 \item $t_0, \ldots, t_{f-1}$ considered as a cyclic list, can be broken up into strings of the form $\pm(-1, p-1, \ldots, p-1, p)$ (where there might be no occurrence of $p-1$) and strings of the form $(0, \ldots, 0)$.
\end{enumerate}
\end{lemma}



\section{Shapes of upper triangular Kisin modules with $k_E$-coefficient-I}

In this section, we study the shape of Kisin modules with $k_E$-coefficient coming from reductions of crystalline representations. We will often use the notations listed below.

\noindent (\textnormal{\textbf{CRYS.}}) Let $p>2$ be an odd  prime, $K/\Qp$ a finite unramified extension.
\begin{itemize}
  \item Suppose $V$ is a crystalline representation of $E$-dimension $d$, such that the labelled Hodge-Tate weights are $\HT_{\kappa_s}(D)=\HT(D_s)=\boldr_s = \{0 = r_{s, 1} < \ldots < r_{s, d} \leq p\}$.
  \item Let $\rho=T$ be a $G_K$-stable $\mathcal O_E$-lattice in $V$, and $\hat \huaM \in \Mod_{\huaS_{\mathcal O_E}}^{\varphi, \Ghat}$ the $(\varphi, \Ghat)$-module attached to $T$. Let $\barrho: =T/\omega_ET$ be the reduction.
   \item Let $\hat \huaM = \Pi_{s=0}^{f-1} \hat{\huaM}_s$ be the decomposition, where $\hat{\huaM}_s=\varepsilon_s \hat\huaM$. And similarly for the ambient Kisin module $\huaM =\Pi_{s=0}^{f-1} \huaM_s$.
  \item Denote $\barhatm$ the reduction modulo $\omega_E$ of $\hat\huaM$, so it decomposes as $\barhatm =\Pi_{s=0}^{f-1} \barhatm_s$. And similarly for the ambient Kisin module $\barm =\Pi_{s=0}^{f-1} \barm_s$.
\end{itemize}

\begin{thm}[{\cite[Thm. 4.22]{GLS14}}] \label{GLS}
With notations from (\textnormal{\textbf{CRYS}}). There exists an $\mathcal O_E[\![u]\!]$-basis $\{e_{s, i}\}_{0 \leq s \leq f-1, 1\leq i \leq d}$ of $\huaM$ such that
\begin{itemize}
\item $\bolde_{s}=(e_{s, 1}, \ldots, e_{s, d})$ is an $\mathcal O_E[\![u]\!]$-basis of $\huaM_s$ for each $s$.
\item We have $\varphi(\bolde_{s-1})=\bolde_{s} X_s \Lambda_s Y_s$ where $X_s, Y_s \in \GL_d(\mathcal O_E[\![u]\!])$ ,$Y_s \equiv Id (\bmod \omega_E)$, and $\Lambda_s  =[E(u)^{r_{s, 1}}, \ldots, E(u)^{r_{s, d}}]$.
\end{itemize}
\end{thm}

\begin{prop}  \label{extKisin}
Let $\barn_i = \barm(\boldt_i; a_i)=\barm(t_{i, 0}, \ldots, t_{i, f-1}; a_i)$ for $1 \leq i \leq d$.
Suppose $\barm \in \Mod_{\huaS_{k_E}}^{\varphi}$ such that $\barm \in \mathcal E(\barn_d, \ldots, \barn_1)$ is an upper triangular extension.
Then there exists basis $\bolde_s=(e_{s, 1}, \ldots, e_{s, d})$ of $\barm_s$, such that
$$\varphi(\bolde_{s-1})=(\bolde_s)A_s=(\bolde_s)
\left(
 \begin{array}{ccccc}
   (a_1)_s u^{t_{1, s}}  & &  x_{s, i, j} \\
    & \ddots  &  \\
    &  & (a_d)_s u^{t_{d, s}}
 \end{array}
\right),
$$
where $A_s$ is an upper triangular matrix such that:
\begin{enumerate}
  \item The diagonal entries in the matrix are $(a_i)_s u^{t_{i, s}}$, $\forall s, \forall i$. The entries on the upper right $x_{s, i, j}$ are polynomials in $k_E[u]$, $\forall s, \forall 1 \leq i <j \leq d$.
  \item For $1 \leq i <j \leq d$, if there does \emph{not} exist nonzero morphism $\barn_j \to \barn_i$, then $\deg (x_{s, i, j}) < t_{j, s}, \forall s$.
  \item  For $1 \leq i <j \leq d$, if there exists nonzero morphism $\barn_j \to \barn_i$, then for any one choice of $s_0$,
   \begin{itemize}[leftmargin=*]
     \item we can make $x_{s_0, i, j}$ into the form $x_{s_0, i, j}  = x_{s_0, i, j}' + a_{s_0, i, j} u^{t_{j, s_0} + \alpha_{s_0}(\barn_j)-\alpha_{s_0}(\barn_i)}$, where $\deg (x_{s_0, i, j}') < t_{j, s_0}$ and $a_{s_0, i, j} \in k_E$.
     \item For all $s \neq s_0$, we still have $\deg (x_{s, i, j}) < t_{j, s}$.
   \end{itemize}
\end{enumerate}
\end{prop}

\begin{proof}
This is easy generalization of \cite[Prop. 7.4]{GLS14}, by induction on $d$.
\end{proof}

\begin{prop} \label{diagonal}
With notations from (\textnormal{\textbf{CRYS}}). Suppose that $\overline \rho$ is upper triangular. Then $\barm$ is upper triangular, i.e., $\barm \in \mathcal E(\barn_d, \ldots, \barn_1)$ where $\barn_i =\barm(t_{i, 0}, \ldots, t_{i, f-1}, a_i)$ are some rank-$1$ Kisin modules with $k_E$-coefficient. Furthermore, for any $s$, we have $\{t_{1, s}, \ldots, t_{d, s}\}=\{r_{s, 1}, \ldots, r_{s, d}\}$ as sets.
\end{prop}

\begin{proof}
$\barm$ is upper triangular by \cite[Lem. 4.4]{Oze13}.
By Theorem \ref{GLS}, there exists basis $\bolde_s$ of $\barm_s$ such that $\varphi(\bolde_{s-1})=\bolde_{s} X_s[u^{r_{s, 1}}, \ldots, u^{r_{s, d}}]$, where $r_{s, 1} < \cdots < r_{s, d}$.

Since $\barm \in \mathcal E(\barn_d, \ldots, \barn_1)$, by Proposition \ref{extKisin}, there exists another basis $\boldf_{s}=(f_{s, i})$ of $\barm_s$ such that $\varphi(\boldf_{s-1})=\boldf_{s}A_s$ where $A_s$ is upper triangular with diagonal elements being
$(a_1)_s u^{t_{1, s}}, \ldots, (a_d)_s u^{t_{d, s}}$.

Suppose $\bolde_s=\boldf_sT_s$ for all $s$, then we will have
$$A_s= T_{s}X_s [u^{r_{s, 1}}, \ldots, u^{r_{s, d}}] \varphi(T_{s-1}^{-1}).$$
Then we can conclude by applying the following lemma, where we let $M=A_s, B=T_sX_s, D=[u^{r_{s, 1}}, \ldots, u^{r_{s, d}}]$ and $A=\varphi(T_{s-1}^{-1})$.
\end{proof}

\begin{lemma}
If we have $M=BDA$, where
\begin{itemize}
  \item $M \in \Mat(k_E\llb u \rrb)$ which is upper triangular with diagonal elements being $c_1u^{t_1}, \ldots, c_du^{t_d}$, where $c_i \in k_E[\![u]\!]^{\times}, \forall i$;
  \item $B \in \GL_d(k_E[\![u]\!])$, $D=[u^{r_{1}}, \ldots, u^{r_{d}}]$ with $0 \leq r_1 \leq \ldots \leq r_d \leq p$, and $A \in \GL_d(k_E\llb u^p \rrb)$ (note here that we do not need $r_i$ to be distinct);
\end{itemize}
then $\{t_1, \ldots, t_d\}=\{r_1, \ldots, r_d\}$ as sets.
\end{lemma}
\begin{proof}
Write $A=(a_{i, j})$, and suppose that $a_{k_1, 1}$ is the top most element in $\col_1(A)$ that is a unit (which exists because $A$ is invertible). Then multiply both sides of $M=B[u^{r_{1}}, \ldots, u^{r_{d}}]A$ by the following invertible upper triangular matrix

$$C=\left(
 \begin{array}{ccccc}
 1 & -\frac{a_{k_1, 2}}{a_{k_1, 1}} & \cdots & -\frac{a_{k_1, d}}{a_{k_1, 1}}\\
  & 1&0 &0\\
  & &\ddots &0\\
 & & &1
 \end{array}
\right)$$

Let $(a_{i, j}')=A'=AC$ (which is still in $\Mat(k_E\llb u^p \rrb)$), then $\col_1(A')=\col_1(A)$, and $a'_{k_1, j}=0$ for $j>1$.
And $M'=MC$ has the same diagonal of $M$. So we can and do assume that we already have $a_{k_1, 1}$ is the top most unit in $\col_1(A)$, and $a_{k_1, j}=0$ for $j>1$.
Now, do the same procedure for the second column of $A$. That is, suppose $a_{k_2, 2}$ is the top most element in $\col_2(A)$ that is a unit, then make $a_{k_2, j}=0$ for $j>2$.
In the end, we can assume that $a_{k_i, i}$ is the top most unit in $\col_i(A)$, and $a_{k_i, j}=0$ for $j>i$. Clearly we have $\{k_1, \ldots, k_d\}=\{1, \ldots, d\}$ as sets.

Then it is clear that $u^{r_{k_i}} \mid \col_i(DA)$ (using the fact that non-units in $k_E\llb u^p \rrb$ are divisible by $u^p$, and $0 \leq r_1 \leq \cdots \leq r_d \leq p$). So $u^{r_{k_i}} \mid \col_i(BDA)$, and we will have $u^{r_{k_i}} \mid u^{t_i}, \forall i$. However, by a determinant argument, $\sum_{i=1}^d r_{k_i} =\sum_{i=1}^d r_i = \sum_{i=1}^d t_i$, so we must have $u^{r_{k_i}} \parallel u^{t_i}, \forall i$, that is $t_i=r_{k_i}, \forall i$.
\end{proof}


\section{Models of upper triangular reductions of crystalline representations}

\newcommand{\WT}{\textnormal{WT}}

Before we can proceed further with the study of shape of upper triangular Kisin modules with $k_E$-coefficient, we need to introduce the condition \textbf{(C-1)}. One of the aims is to make sure that:
\begin{itemize}
  \item When $\barm \in \mathcal E(\barn_d, \ldots, \barn_1)$, there does \emph{not} exist nonzero morphism $\barn_j \to \barn_i$ for any $1 \leq i <j \leq d$ (i.e., the situation in Statement (3) of Proposition \ref{extKisin} does not happen).
\end{itemize}
However note that \textbf{(C-1)} is stronger than the bullet condition above. We will need the full strength of \textbf{(C-1)} in Theorem \ref{lifting-2}.

\begin{defn}
 \begin{enumerate}
   \item For a rank-1 module $\barm= \overline{\huaM}(t_0, \ldots, t_{f-1}; a) \in \Mod_{\huaS_{k_E}}^{\varphi}$, define $\WT(\barm)$ as the \emph{ordered} set $\{t_0, \ldots, t_{f-1}\}$.

    \item For an upper triangular module $\barm \in \mathcal E(\barn_d, \ldots, \barn_1)$, define $\WT(\barm)$ to be the $d\times f$-matrix, where $\row_i (\WT(\barm))= \WT(\barn_i), \forall 1 \leq i \leq d$.
 \end{enumerate}
\end{defn}

\begin{defn}
Let
\begin{itemize}
  \item $\barzeta_1, \ldots, \barzeta_d : G_K \to k_E^{\times}$ be $d$ characters.
  \item $h_0, \ldots, h_{f-1}$ be $f$ sets, where $h_s$ is a set of $d$ distinct integers in $[0, p]$, for each $0 \leq s \leq f-1$.
\end{itemize}

A \textnormal{\textbf{model}} of the ordered sequence $\{\barzeta_1, \ldots, \barzeta_d\}$ with respect to the ordered sequence $\{h_0, \ldots, h_{f-1}\}$ is a $d\times f$-matrix
$N=(n_{i, s})_{1\leq i\leq d, 0 \leq s \leq f-1},$
such that
\begin{itemize}
  \item  $\col_s(N)=h_s$ as sets of numbers, for $0\leq s \leq f-1$;
  \item For each $1\leq i \leq d$, there exists a rank-$1$ Kisin module with $k_E$-coefficient defined by $\barm(n_{i, 0}, \ldots, n_{i, f-1}; a_i)$ for some $a_i \in k_E^{\times}$ such that
      $$T_{\huaS}(\barm(n_{i, 0}, \ldots, n_{i, f-1}; a_i))=\barzeta_{i} \mid_{G_{\infty}}.$$
\end{itemize}
\end{defn}

With notations from (\textnormal{\textbf{CRYS}}), suppose $\overline \rho$ is upper triangular, that is, $\overline \rho \in \mathcal{E}(\overline \chi_1, \ldots, \overline \chi_d)$ for some characters. By Proposition \ref{diagonal}, there exists some rank-$1$ Kisin modules $\barn_1, \ldots, \barn_d$, such that $\barm \in \mathcal{E}(\barn_d, \ldots, \barn_1)$ and $T_{\huaS}(\barm) =\overline \rho \mid_{G_{\infty}}$. Suppose $\barn_i =\barm(t_{i, 0}, \ldots, t_{i, f-1}; a_i)$. Then $a_i$ are uniquely determined, and
$\col_s (\textnormal{WT}(\barm))=\{t_{1, s}, \ldots, t_{d, s}\}$ is equal to $\HT(D_s)$ as sets of numbers. So, the matrix $\WT(\barm)$ is a model of $\{\overline \chi_1, \ldots, \overline \chi_d\}$ with respect to $\{\HT_0(D), \ldots, \HT_{f-1}(D)\}$. For many theorems in our paper, we will need to have the following condition.

\smallskip
\noindent \textnormal{\textbf{Condition (C-1)}}: Suppose $\overline{\rho} \in \mathcal E(\barchi_1, \ldots, \barchi_d)$, then $\{\overline \chi_1, \ldots, \overline \chi_d\}$ has a unique model with respect to $\{\HT_0(D), \ldots, \HT_{f-1}(D)\}$.
\smallskip

\begin{rem} \label{3nothappen}
It is clear that when condition \textnormal{\textbf{(C-1)}} is satisfied, $\barchi_i \neq \barchi_j, \forall i \neq j$. So in particular, the situation in Statement (3) of Proposition \ref{extKisin} will not happen.
\end{rem}

Here are some examples when the condition is satisfied.
\begin{lemma}
The condition \textnormal{\textbf{(C-1)}} is satisfied if one of the following is true,
 \begin{enumerate}
 \item $K=\Qp$, i.e., $f=1$, and the differences between any two elements in $HT(D_0)$ are never $p-1$.
  \item For each $s$, the differences between two elements in $\HT(D_s)$ are never $1$. And for one $s_0$, $p-1 \notin \HT(D_{s_0})$.
  \item For each $s$, $\HT(D_s) \subseteq [0, p-1]$. And for one $s_0$, $p-1 \notin \HT(D_{s_0})$.
\end{enumerate}
\end{lemma}
\begin{proof}
If $N$ is another model other than $\WT(\barm)$, then by Lemma \ref{comparerank1}, for each $i$, $\row_i(N)-\row_i(T)$ will satisfy the solutions in Lemma \ref{solution}.
\end{proof}

The naming of the concept of \emph{model} reflects our initial intention to generalize results in \cite[\S 8.2]{GLS14} and \cite[\S 5.3]{GLS13}. In particular, we wanted to find some \emph{maximal model}, which will help us to prove an analogue of \cite[Prop. 5.3.4]{GLS13}. Unfortunately, we were not able to achieve this.


\section{Shapes of upper triangular Kisin modules with $k_E$-coefficient-II}

\begin{prop} \label{shape}
With notations from (\textnormal{\textbf{CRYS}}) and Proposition \ref{diagonal}. Suppose that $\overline \rho$ is upper triangular, and there does \emph{not} exist nonzero morphism $\barn_j \to \barn_i$ for any $1 \leq i <j \leq d$ (e.g., when Condition \textnormal{\textbf{(C-1)}} is satisfied).

Let $\bolde_s$ be a basis of $\barm_s$ as in Proposition \ref{extKisin}, such that $\varphi(\bolde_{s-1})=(\bolde_s)A_s$ where $A_s$ satisfies the statements of Proposition \ref{extKisin} (note that Statement (3) of Proposition \ref{extKisin} will not happen).
Then we must have $x_{s, i, j}=u^{t_{i, s}}y_{s, i, j}$, where
\begin{itemize}
  \item $y_{s, i, j}=0$ if $t_{j, s} < t_{i, s}$.
  \item $y_{s, i, j} \in k_E$ if $t_{j, s}>t_{i, s}$.
\end{itemize}
\end{prop}

\begin{rem} \label{rem: upper bound of shape}
We remark that Proposition \ref{shape} effectively gives an ``upper bound" for the shape of upper triangular Kisin modules with $k_E$-coefficient that we are studying (we mentioned about this ``upper bound" in the Introduction). This is because for the matrices $A_s$, the elements that can vary are those $y_{s, i, j}$ when $t_{j, s} > t_{i, s}$, and they can only vary in $k_E$. We will need Section 5 to give precise meaning for the ``upper bound", see Proposition \ref{shapesubmodule}.
\end{rem}

\begin{proof}
From the proof of Proposition \ref{diagonal}, we have
$A_s\varphi(T_{s-1})= T_{s}X_s [u^{r_{s, 1}}, \ldots, u^{r_{s, d}}] .$
Let $R_s \in \GL_d(k_E)$ such that
$R_s^{-1}[u^{r_{s, 1}}, \ldots, u^{r_{s, d}}] R_s =[u^{t_{1, s}}, \ldots, u^{t_{d, s}}],$
and consider the equality
$$A_s\varphi(T_{s-1})R_s= T_{s}X_sR_s [u^{t_{1, s}}, \ldots, u^{t_{d, s}}].$$

The $i$-th column on the right hand side is divisible by $u^{t_{i, s}}$. Let $\varphi(T_{s-1})=P_{s-1} +u^pQ_{s-1}$ where $P_{s-1} \in \GL_d(k_E), Q_{s-1} \in \Mat_d(k_E\llb u^p\rrb)$, so we will have $u^{t_{i, s}}\mid \col_i(A_{s}P_{s-1}R_s)$.
Then we can apply the following lemma to conclude, where we let $X=A_s[(a_1)_s^{-1}, \ldots, (a_d)_s^{-1}]$, and $A=[(a_1)_s, \ldots, (a_d)_s]P_{s-1}R_s$ in the lemma.
\end{proof}

\begin{lemma} \label{shapelemma}
Suppose $t_1, \ldots, t_d$ are distinct integers in $[0, p]$.
Suppose $$X = \left(
 \begin{array}{ccccc}
 u^{t_1} & & x_{i,j}\\
         & \ddots & \\
         & & u^{t_d}
 \end{array}
\right), \quad A \in \GL_d(k_E),$$

\noindent where $X$ is an upper triangular matrix with coefficients in $k_E[u]$, such that
\begin{itemize}
  \item $\deg(x_{i, j}) < t_j$, and
  \item $u^{t_i} \mid \col_i(XA)$.
\end{itemize}
Then we must have
$x_{i, j}=u^{t_i}y_{i, j},$ where
\begin{itemize}
  \item $y_{i, j}=0$ if $t_{j} < t_{i}$, and
  \item  $y_{i, j} \in k_E$ if $t_{j}>t_{i}$.
\end{itemize}
\end{lemma}

\begin{proof}
We prove the lemma by induction on the dimension $d$.
\begin{itemize}
  \item We say that an upper triangular matrix $X \in \Mat(k_E[u])$ of the shape
  $\left(
 \begin{array}{ccccc}
 u^{t_1} & & x_{i,j}\\
         & \ddots & \\
         & & u^{t_d}
 \end{array}
\right)$
 satisfies the property $(\rm DEG)$ if $\deg(x_{i, j}) < t_j, \forall i<j$.
  \item If the conclusion of the lemma is satisfied, we say that $X$ satisfies property $(P)$. We also say $x_{i, j}$ satisfies $(P)$ for a single index $(i, j)$ if $x_{i, j}$ satisfies the conclusion of the lemma.
\end{itemize}

The lemma is trivially true when $d=1$. We want to remark here that when $d=2$, the lemma is true by arguments in \cite[Thm. 7.9]{GLS14}. However, here we give a general argument, which will work for all $d$. So now suppose the lemma is true when the dimension is less than $d$. We now prove it when the dimension becomes $d$. In order to do so, we first prove two sublemmas (Sublemma \ref{sublemma1} and Sublemma \ref{sublemma2}), which are indeed special cases when the dimension becomes $d$. The reason that we are writing these two special cases first, is because they will make the general process much more transparent.
\end{proof}

\noindent \textbf{Notations}. We will use $\Mat(x_{i, j})$ to mean the matrix where the only nonzero element is its $i$-th row, $j$-th column element, and the element is precisely $x_{i, j}$. We hope this does not cause confusion. For a matrix $A$, we use $A_{i, j}$ to mean the co-matrix of $a_{i, j}$, that is, the matrix after deleting $i$-th row and $j$-th column of $A$.

Before we prove the sublemmas, we make a useful definition. Let $X$ satisfy (DEG). We call the following procedure an \textbf{allowable procedure} for $X$:
$$X \rightsquigarrow X'=X(Id - \Mat(c_{i, j})),$$
where $1\leq i<j \leq d$ are two numbers such that $t_i <t_j$, and $c_{i, j} \in k_E$.

It is easy to see that if we let $A'=(Id - \Mat(c_{i, j}))^{-1}A$, then we have the following (note that the only change is the $j$-th column of $X$, and using that $t_i<t_j$)
\begin{itemize}
  \item $X'$ still satisfies property $\rm{(DEG)}$.
  \item $X'A'=XA$ (so in particular $u^{t_i} \mid \col_i(X'A') \Leftrightarrow u^{t_i} \mid \col_i(XA), \forall i$).
  \item $X$ satisfies $(P)$ if and only if $X'$ satisfies $(P)$.
\end{itemize}

\begin{rem} \label{change}
A very useful remark is that, when $X$ satisfies (P), one can apply finite times of \emph{allowable procedures} to change $X$ to the diagonal matrix $[u^{t_1}, \ldots, u^{t_d}]$. One can start by making $x_{d-1, d}=0$ by letting
$X \rightsquigarrow  X(Id - \Mat(y_{d-1, d})).$
And then, one can consecutively make $x_{d-2, d}=0, \ldots, x_{1, d}=0$. Then one can change $x_{d-2, d-1}$ to $0$, and so on.
\end{rem}

\begin{sublemma} \label{sublemma1}
If $t_d$ is maximal in $\{t_1, \ldots, t_d\}$, then $X$ satisfies $(P)$.
\end{sublemma}

\begin{proof}
Since $u^{t_d} \mid X  \left(
 \begin{array}{ccccc}
 a_{1, d}  \\
    \vdots     \\
  a_{d, d}
 \end{array}
\right), $
so we have
$u^{t_d} \mid u^{t_{d-1}}a_{d-1, d} + x_{d-1, d}a_{d, d}.$ Since $\deg(u^{t_{d-1}}a_{d-1, d} + x_{d-1, d}a_{d, d})<t_d$, so $u^{t_{d-1}}a_{d-1, d} + x_{d-1, d}a_{d, d}=0$. We claim that $a_{d, d} \neq 0$. Suppose otherwise, then $a_{d-1, d}=0$. Since
$$u^{t_d} \mid u^{t_{d-2}}a_{d-2, d} + x_{d-2, d-1} a_{d-1, d} + x_{d-2, d} a_{d, d},$$
so $a_{d-2, d}=0$. And similarly we will find $a_{i, d}=0, \forall 1 \leq i \leq d$, which is impossible since $A$ is invertible.

So now $a_{d, d} \neq 0$, we must have
$x_{d-1, d} =u^{t_{d-1}}y_{d-1, d}$ for some $y_{d-1, d} \in k_E.$
Now let $X' = X (Id - \Mat(y_{d-1, d}))$ and $A'= (Id - \Mat(y_{d-1, d}))^{-1}A,$
so $x_{d-1, d}'=0$. Note that the above procedure is an allowable procedure, so we can and do assume our $X$ already satisfies that $x_{d-1, d}=0$. Then we have $a_{d-1, d}=0$.

Since $u^{t_d} \mid u^{t_{d-2}}a_{d-2, d} + x_{d-2, d-1} a_{d-1, d} + x_{d-2, d} a_{d, d}= u^{t_{d-2}}a_{d-2, d} + x_{d-2, d} a_{d, d},$
by similar argument as above, $x_{d-2, d} =u^{t_{d-2}}y_{d-2, d}$. And then we can change $X$ to $X'= X (Id -\Mat(y_{d-2, d}))$, and argue similarly as above.
So in the end, we can actually assume that $x_{i, d}=0, a_{i, d}=0$ for $1 \leq i \leq d-1$. So we have
$$ X A = \left(
 \begin{array}{ccccc}
 X_{d, d} & 0  \\
 0 & u^{t_d}
 \end{array}
\right)
\left(
 \begin{array}{ccccc}
 A_{d, d} & 0  \\
 a_{d, j}\mid_{1\leq j \leq d-1} & a_{d, d}
 \end{array}
\right)
=
\left(
 \begin{array}{ccccc}
 X_{d, d}A_{d, d} & 0  \\
 u^{t_d}a_{d, j}\mid_{1\leq j \leq d-1} & u^{t_d}a_{d, d}
 \end{array}
\right).
$$

So we will have that $u^{t_k} \mid \col_k(X_{d, d}A_{d, d})$ for $1 \leq k \leq d-1$. By induction hypothesis, $X_{d, d}$ satisfies $(P)$, so $X$ satisfies $(P)$, and we are done.
\end{proof}

\begin{sublemma} \label{sublemma2}
If $t_1$ is maximal in $\{t_1, \ldots, t_d\}$, then $X$ satisfies $(P)$.
\end{sublemma}

\begin{proof}
Now  $u^{t_1} \mid X\left(
 \begin{array}{ccccc}
 a_{1, 1}\\
 \vdots\\
 a_{d, 1}
 \end{array}
\right),$
so we have $u^{t_1} \mid u^{t_d}a_{d, 1}$, and so $a_{d, 1}=0$. Then similarly we have $a_{i, 1}=0$ for $2 \leq i \leq d$.
So $$ X A = \left(
 \begin{array}{ccccc}
 u^{t_1} & (x_{1, j})_{2 \leq j \leq d}\\
 0 & X_{1, 1}
 \end{array}
\right)
\left(
 \begin{array}{ccccc}
 a_{1, 1} & (a_{1, j})_{2 \leq j \leq d}  \\
 0 & A_{1, 1}
 \end{array}
\right)
=
\left(
 \begin{array}{ccccc}
 u^{t_1}a_{1, 1} & u^{t_1}(a_{1, j})+(x_{1, j})A_{1, 1}  \\
 0 & X_{1, 1}A_{1, 1}
 \end{array}
\right).
$$

Then we will have that $u^{t_{k+1}} \mid \col_k (X_{1, 1}A_{1, 1})$ for $1 \leq k \leq d-1$.
So we can use induction hypothesis to see that $X_{1, 1}$ satisfies $(P)$.
What is left is to show that $x_{1, j}=0$ for $2 \leq j \leq d$.
Since $X_{1, 1}$ satisfies (P), by Remark \ref{change}, we can apply finite steps of allowable procedures on $X$ (these procedures do not involve the first column of $X$),
 so that $X_{1, 1}$ becomes a diagonal matrix. That is, we can and do assume $X_{1, 1}=[u^{t_2}, \ldots, u^{t_d}].$

Suppose $t_{k_1} = \max \{t_2, \ldots, t_d\}$, so we have
$u^{t_{k_1}} \mid X\left(
 \begin{array}{ccccc}
 a_{1, k_1}\\
 \vdots\\
 a_{d, k_1}
 \end{array}
\right).$
Because $X_{1, 1}$ is diagonal, it is easy to see that we must have $a_{i, k_1}=0$, for $i \neq 1, k_1$. So we now have
$$u^{t_{k_1}} \mid X \col_{k_1}(A)=
\left(
 \begin{array}{ccccc}
 u^{t_1} & (x_{1, j})\\
 0 & X_{1, 1}
 \end{array}
\right)
\left(\begin{array}{ccccc}
 a_{1, k_1}\\
 0\\
 \vdots\\
 0\\
 a_{k_1, k_1}\\
 0\\
 \vdots\\
 0
 \end{array}
  \right)
=
\left(
\begin{array}{ccccc}
u^{t_1}a_{1, k_1} + x_{1, k_1}a_{k_1, k_1}\\
 0\\
 \vdots\\
 0\\
 u^{t_{k_1}}a_{k_1, k_1}\\
 0\\
 \vdots\\
 0
 \end{array}
 \right).
   $$

Since $\col_1(A)=(1, 0, \ldots, 0)^{T}$ and $A$ is invertible, we must have $a_{k_1, k_1} \neq 0$, and so $x_{1, k_1}=0$.

Now, suppose $t_{k_2} =\textnormal{max} \{\{t_1, \ldots, t_d\}-\{t_1, t_{k_1}\}\}$. Then similarly we can see that $a_{i, k_2}=0$ for $i \neq 1, k_1, k_2$. We must have $a_{k_2, k_2} \neq 0$ because $A$ is invertible. We have
$u^{t_{k_2}} \mid u^{t_1}a_{1, k_2} +  x_{1, k_1}a_{k_1, k_2} + x_{1, k_2}a_{k_2, k_2} = u^{t_1}a_{1, k_2}+ x_{1, k_2}a_{k_2, k_2},$
and we can conclude that $x_{1, k_2}=0$. Argue similarly with consecutive next maximal elements in $\{t_1, \ldots, t_d\}$ will show that $x_{1, j}=0$ for $2 \leq j \leq d$.
\end{proof}

\begin{proof}[\textbf{Proof of Lemma \ref{shapelemma}, continued.}]
So now, let us prove the general lemma. Assume $t_k =\max \{t_1, \ldots, t_d\}$, by Sublemma \ref{sublemma1} and Sublemma \ref{sublemma2}, we can assume $1 <k <d$. So we have
$u^{t_k} \mid X\left(
 \begin{array}{ccccc}
 a_{1, k}\\
 \vdots\\
 a_{d, k}
 \end{array}
\right).$
Then $a_{i, k}=0$ for $k < i \leq d$. So we have that
$u^{t_k} \mid \left(
 \begin{array}{ccccc}
 u^{t_1} & & x_{i,j}\\
         & \ddots & \\
         & & u^{t_k}
 \end{array}
\right)
\left(
 \begin{array}{ccccc}
 a_{1, k}\\
 \vdots\\
 a_{k, k}
 \end{array}
\right).
$
Apply similar allowable procedures as in Sublemma \ref{sublemma1}, we can make $x_{i, k}=0, a_{i, k}=0$ for $1 \leq i \leq k-1$. So,
$$XA=
\left(
 \begin{array}{ccccc}
X_1 & 0 & X_2 \\
0 & u^{t_k} & X_3\\
0 & 0 & X_4
 \end{array}
\right)
\left(
 \begin{array}{ccccc}
A_1 & 0 & A_2 \\
A_3 & a_{k, k} & A_4 \\
A_5 & 0 & A_6
 \end{array}
\right)
=
\left(
 \begin{array}{ccccc}
X_1A_1+X_2A_5 & 0 & X_1A_2+X_2A_6 \\
u^{t_k}A_3+X_3A_5 & u^{t_k}a_{k, k} & u^{t_k}A_4+X_3A_6\\
X_4A_5 & 0 & X_4A_6
 \end{array}
\right).
$$

It is easy to see that $X_{k, k}$ satisfies $(P)$ by induction hypothesis. What is left is to show that $x_{k, j}=0$ for $j>k$. It can be done similarly as in Sublemma \ref{sublemma2}.
\end{proof}


\section{$\mathcal O_E$-module structure of extension classes}

In this section, we prove that certain sets of extension classes have natural $\mathcal O_E$-module structures, and that these structures are compatible with each other. Many results in this section are clearly valid for general $K/\Qp$ and any prime number $p$. But for our purpose, we assume throughout that $K/\Qp$ is unramified and $p>2$. \emph{However}, sometimes we specifically point out the assumption $p>2$, just to emphasize the necessity.

We want to remark here that this section is heavily influenced by the unpublished notes \cite{GLS14+} of Toby Gee, Tong Liu and David Savitt. In fact, practically all the major definitions and results (in particular, Propositions \ref{modulestr}, \ref{exttorep}, \ref{tensorE}) are taken directly from \cite{GLS14+}, including the proofs. We want to heartily thank their generosity again. The notes \cite{GLS14+} played a great and essential role in shaping the style of the main local results in Section 7, and in fact has corrected a quite serious mistake in an earlier draft of our paper.

\subsection{Extension of Kisin modules and $(\varphi, \Ghat)$-modules}

\begin{defn}\hfill
\begin{enumerate}
  \item A sequence $0 \to \huaL \to \huaN \to \huaL' \to 0$ in $\Mod_{\huaS_{\mathcal O_E}}^{\varphi}$
  is called short exact, if it is short exact as a sequence of $\huaS_{\mathcal O_E}$-modules.
  \item A sequence $0 \to \hat\huaL \to \hat\huaN \to \hat\huaL' \to 0$ in $\Mod_{\huaS_{\mathcal O_E}}^{\varphi, \Ghat}$
  is called short exact, if it is short exact as a sequence of $\huaS_{\mathcal O_E}$-modules.
  \item We can define short exact sequences in $\Mod_{\huaS_{k_E}}^{\varphi}$ and $\Mod_{\huaS_{k_E}}^{\varphi, \Ghat}$ analogously.
\end{enumerate}
\end{defn}

\begin{defn}\hfill
\begin{enumerate}
  \item Suppose $\huaL, \huaL' \in \Mod_{\huaS_{\mathcal O_E}}^{\varphi}$.
  Let $\Ext(\huaL', \huaL)$ be the set of short exact sequences $0 \to \huaL \to \huaN \to \huaL' \to 0$ in the category $\Mod_{\huaS_{\mathcal O_E}}^{\varphi}$, modulo the equivalence relation as follows. Call
  $0 \to \huaL \to \huaN^{(1)} \to \huaL' \to 0$
  and
  $0 \to \huaL \to \huaN^{(2)} \to \huaL' \to 0$
  equivalent, if there exists $\xi: \huaN^{(1)} \to \huaN^{(2)}$ such that the following diagram commutes:
$$\begin{CD}
0 @>>> \huaL @>>> \huaN^{(1)} @>>> \huaL' @>>> 0\\
@. @| @V\xi VV @| @.\\
0 @>>> \huaL @>>> \huaN^{(2)} @>>> \huaL' @>>> 0
\end{CD}$$

\item We can define similar $\Ext$'s for pairs of objects in categories $\Mod_{\huaS_{\mathcal O_E}}^{\varphi, \Ghat}$, $\Mod_{\huaS_{k_E}}^{\varphi}$, and $\Mod_{\huaS_{k_E}}^{\varphi, \Ghat}$.
\end{enumerate}
\end{defn}

\begin{rem}
Indeed, when we define $\Ext$'s in various categories, we should have added certain subscripts to distinguish the situation. However, in our paper, the category where we are taking $\Ext$ are mostly clear from the context. We do sometimes add some subscripts (whose meaning will be obvious) to avoid ambiguity.
\end{rem}

\begin{prop}[\cite{GLS14+}] \label{modulestr}
The following statements hold.
\begin{enumerate}
  \item Suppose $\huaL', \huaL \in \Mod_{\huaS_{\mathcal O_E}}^{\varphi}$, then $\Ext (\huaL', \huaL)$ has an $\mathcal O_E$-module structure.
  \item Suppose $\hat\huaL', \hat\huaL \in \Mod_{\huaS_{\mathcal O_E}}^{\varphi, \Ghat}$, then $\Ext(\hat \huaL', \hat \huaL)$ has an $\mathcal O_E$-module structure.

  \item Suppose $p>2$, and both $\hatL', \hatL \in \Mod_{\huaS_{\mathcal O_E}}^{\varphi, \Ghat}$ are crystalline, then $\Ext_{\rm{cris}}(\hatL', \hatL)$ has an $\mathcal O_E$-module structure. Here, $\Ext_{\rm{cris}}(\hatL', \hatL) \subseteq \Ext(\hatL', \hatL)$ consists of equivalence classes of short exact sequences where the central object is crystalline.

  \item Suppose $\barl', \barl \in \Mod_{\huaS_{k_E}}^{\varphi}$, then $\Ext(\barl', \barl)$ has a $k_E$-vector space structure.
   \item Suppose $\barhatl', \barhatl \in \Mod_{\huaS_{k_E}}^{\varphi, \Ghat}$, then $\Ext (\barhatl', \barhatl)$ has a $k_E$-vector space structure.

\item Suppose $\huaL', \huaL \in \Mod_{\huaS_{\mathcal O_E}}^{\varphi}$, then the following natural map
    $$\Ext (\huaL', \huaL)/\omega_E\Ext (\huaL', \huaL) \to \Ext (\barl', \barl) $$
    is an injective homomorphism of $k_E$-vector spaces.

\item Suppose $\hat\huaL', \hat\huaL \in \Mod_{\huaS_{\mathcal O_E}}^{\varphi, \Ghat}$, then the following natural map
    $$\Ext(\hat \huaL', \hat \huaL)/\omega_E\Ext(\hat \huaL', \hat \huaL) \to \Ext (\barhatl', \barhatl)$$
    is an injective homomorphism of $k_E$-vector spaces.

\end{enumerate}
\end{prop}

\begin{proof}
\textbf{Proof of (1)}. Let $d= \textnormal{rk}_{\huaSOE}\huaL, d'= \textnormal{rk}_{\huaSOE}\huaL'$.
Let $\bolde=(e_1, \ldots, e_d)$ be a fixed $\huaS_{\mathcal O_E}$-basis of $\huaL$, and let $\bolde'=(e_1', \ldots, e'_{d'})$ be a fixed $\huaS_{\mathcal O_E}$-basis of $\huaL'$.
\emph{For brevity}, in the following, we simply use $e, e'$ to denote the bases. Suppose $\varphi(e)=eA, \varphi(e')=e'A'$, where $A, A' \in \Mat(\huaS_{\mathcal O_E})$, and let $B, B' \in \Mat(\huaS_{\mathcal O_E})$ such that $AB=(E(u))^r Id, A'B'=(E(u))^r Id$.

Let $M$ be the set of matrices in $\Mat(\huaS_{\mathcal O_E})$ with the shape
$   \left(
 \begin{array}{ccccc}
    A  & C \\
0 & A'
 \end{array}
\right)$ such that there exists
$\left(
 \begin{array}{ccccc}
    B & D\\
0 & B'
 \end{array}
\right) \in \Mat(\huaS_{\mathcal O_E})
$ such that
$
 \left(
 \begin{array}{ccccc}
    A  & C \\
0 & A'
 \end{array}
\right)
\left(
 \begin{array}{ccccc}
    B & D\\
0 & B'
 \end{array}
\right)
=(E(u))^r Id.$
$M$ has a natural $\mathcal O_E$-module structure where
$$a\left(
 \begin{array}{ccccc}
    A & C_1\\
0 & A'
 \end{array}
\right)
+
b\left(
 \begin{array}{ccccc}
    A & C_2\\
0 & A'
 \end{array}
\right)
:=
\left(
 \begin{array}{ccccc}
    A & aC_1+bC_2\\
0 & A'
 \end{array}
\right), \forall a, b \in \mathcal O_E.$$

Define an equivalence relation in $M$ such that
$\left(
 \begin{array}{ccccc}
    A & C_1\\
0 & A'
 \end{array}
\right)$ and
$\left(
 \begin{array}{ccccc}
    A & C_2\\
0 & A'
 \end{array}
\right)$ are equivalent if there exists a matrix $W$, such that $C_1 -C_2 =WA'-A\varphi(W)$.
Let $M_0$ be the subset of $M$ consisting of elements equivalent to $\left(
 \begin{array}{ccccc}
    A & 0\\
0 & A'
 \end{array}
\right)$. One can easily check that $M_0$ is a submodule of $M$. Let $\tilde M := M/M_0$. One can easily check that if we change $\bolde, \bolde'$ to some other bases, we will get isomorphic $\tilde M$.


Now for $\tilde x \in \Ext (\huaL', \huaL)$, choose a representative: $0 \to \huaL \to \huaN \to \huaL' \to 0$. Take a section (of $\huaS_{\mathcal O_E}$-modules) $s_{\tilde x}: \huaL' \to \huaN$, then $(e, s_{\tilde x}(e'))$ forms a basis for $\huaN$.
Then
$\varphi(e, s_{\tilde x}(e')) = (e, s_{\tilde x}(e'))
\left(
 \begin{array}{ccccc}
    A & C_{\tilde x, s_{\tilde x}}\\
0 & A'
 \end{array}
\right)$
for some matrix $C_{\tilde x, s_{\tilde x} }.$

We define a map of sets $F: \Ext(\huaL', \huaL) \to \tilde M$ by mapping $\tilde x$ above to $\left(
 \begin{array}{ccccc}
    A & C_{\tilde x, s_{\tilde x}}\\
0 & A'
 \end{array}
\right)$. One can easily check that this map is well-defined, in particular, it does not depend on the choice of the short exact sequence for $\tilde x$, or the choice of the section $s_{\tilde x}$. One can also easily check that this map is a bijection. So we can equip $\Ext (\huaL', \huaL)$ with an $\mathcal O_E$-structure via that on $\tilde M$.

\textbf{Proof of (2)}. The proof is very similar to (1). We will give a sketch, since the ideas will be used later.
Again, let $e, e'$ be a basis of $\huaL, \huaL'$ respectively, and suppose $\varphi(e)=eA, \varphi(e')=e'A'$, $g(1\otimes_{\varphi}e)=(1\otimes_{\varphi}e)X_g$,  $g(1\otimes_{\varphi}e')=(1\otimes_{\varphi}e')X_g'$, where $X_g, X'_g \in \Mat(\hat R_{\mathcal O_E}), \forall g \in \hat G$.

Let $M$ be the set where an element $m \in M$ is a set of matrices consisting of
$m_{\varphi}=\left(
 \begin{array}{ccccc}
    A & C\\
0 & A'
 \end{array}
\right)$ and
$m_g=\left(
 \begin{array}{ccccc}
    X_g & Y_g\\
0 & X_g'
 \end{array}
\right)
$ for each $g \in \hat G$,
where $C \in \Mat(\huaS_{\mathcal O_E}), Y_g \in \Mat(\hat R_{\mathcal O_E})$, which satisfy the following conditions:
\begin{itemize}
  \item There exists $\left(
 \begin{array}{ccccc}
    B & D\\
0 & B'
 \end{array}
\right)
\in \Mat(\huaS_{\mathcal O_E})$ such that
$
 \left(
 \begin{array}{ccccc}
    A  & C \\
0 & A'
 \end{array}
\right)
\left(
 \begin{array}{ccccc}
    B & D\\
0 & B'
 \end{array}
\right)
= (E(u))^r Id.
$

\item $m_{g_1g_2} = g_1(m_{g_2}) m_{g_1}, \forall g_1, g_2 \in \hat G$.

\item $m_g g(\varphi(m_{\varphi})) = \varphi(m_{\varphi}) \varphi(m_g), \forall g \in \hat G$.

\item $m_h =\left(
 \begin{array}{ccccc}
    Id & 0\\
0 & Id
 \end{array}
\right), \forall h \in H_K.$

\item $Y_g \in \Mat((I_{+}\hat R)_{\mathcal O_E}), \forall g \in \hat G$.
\end{itemize}

$M$ has a natural $\mathcal O_E$-module structure, where if $m^{(i)} \in M, i=1, 2$ such that
$m^{(i)}_{\varphi}=\left(
 \begin{array}{ccccc}
    A & C^{(i)}\\
0 & A'
 \end{array}
\right),
m^{(i)}_g=\left(
 \begin{array}{ccccc}
    X_g & Y_g^{(i)}\\
0 & X_g'
 \end{array}
\right)_{g \in \hat G},
$
and if $a, b \in \mathcal O_E$, then define $m=am^{(1)}+b m^{(2)}$ to be such that
$$m_{\varphi}=\left(
 \begin{array}{ccccc}
    A & aC^{(1)}+bC^{(2)}\\
0 & A'
 \end{array}
\right),
m_g=\left(
 \begin{array}{ccccc}
    X_g & aY_g^{(1)}+bY_g^{(2)}\\
0 & X_g'
 \end{array}
\right)_{g \in \hat G}.
$$

Define an equivalence relation on $M$, where $m^{(1)}$ and $m^{(2)}$ are equivalent, if there exists a matrix $W \in \Mat(\huaS_{\mathcal O_E})$, such that $C^{(1)}-C^{(2)}=WA'-A\varphi(W)$ and $Y_g^{(1)}-Y_g^{(2)}=W X_g' -X_g g(W), \forall g \in \hat G.$
Then let $M_0$ be the submodule of $M$ consisting of elements equivalent to $m_0$, where
$m_{0, \varphi}=\left(
 \begin{array}{ccccc}
    A & 0\\
0 & A'
 \end{array}
\right),
m_{0, g}=\left(
 \begin{array}{ccccc}
    X_g & 0\\
0 & X_g'
 \end{array}
\right)_{g \in \hat G}.
$
Let $\tilde M =M/M_0$. One can show similarly as in the proof of statement (1) that there is a bijection between $\Ext(\hat \huaL', \hat \huaL)$ and $\tilde M$, and so one can equip an $\mathcal O_E$-module structure on $\Ext(\hat \huaL', \hat \huaL)$.

\textbf{Proof of (3)}. Similarly as in the proof of (2), let $e, e'$ be a basis of $\huaL, \huaL'$ respectively, and suppose $\varphi(e)=eA, \varphi(e')=e'A'$, $\tau(1\otimes_{\varphi}e)=(1\otimes_{\varphi}e)X_{\tau}$,  $\tau(1\otimes_{\varphi}e')=(1\otimes_{\varphi}e')X_{\tau}'$, where
$X_{\tau}-Id, X'_{\tau}-Id \in \Mat(\Rhat_{\mathcal O_E} \cap (u^p\varphi(\mathfrak t)W(R)_{\mathcal O_E}))$, by Theorem \ref{Ozeki} and Remark \ref{remarkOzeki}.

Let $M$ be the set where an element $m \in M$ is a set of two matrices
$m_{\varphi}=\left(
 \begin{array}{ccccc}
    A & C\\
0 & A'
 \end{array}
\right),
m_{\tau}=\left(
 \begin{array}{ccccc}
    X_{\tau} & Y_{\tau}\\
0 & X_{\tau}'
 \end{array}
\right),
$
where $C \in \Mat(\huaS_{\mathcal O_E}), Y_{\tau} \in \Mat(\Rhat_{\mathcal O_E})$, which satisfy the following conditions:
\begin{itemize}
  \item There exists $\left(
 \begin{array}{ccccc}
    B & D\\
0 & B'
 \end{array}
\right)
$ such that
$
 \left(
 \begin{array}{ccccc}
    A  & C \\
0 & A'
 \end{array}
\right)
\left(
 \begin{array}{ccccc}
    B & D\\
0 & B'
 \end{array}
\right)
= (E(u))^r Id.
$

\item $m_{\tau} \tau(\varphi(m_{\varphi})) = \varphi(m_{\varphi}) \varphi(m_\tau)$.

\item   $Y_{\tau} \in  \Mat(\Rhat_{\mathcal O_E} \cap (u^p\varphi(\mathfrak t)W(R)_{\mathcal O_E}))$.
\item $g(m_{\tau}) = \prod_{k=0}^{\varepsilon_p(g)-1} \tau^k(m_{\tau})$ for all $g\in G_{\infty}$ such that $\varepsilon_p(g) \in \mathbb{Z}^{\geq 0}$.
\end{itemize}

Similarly as in the proof of Statement (2), $M$ has a natural $\mathcal O_E$-module structure. Then we can similarly define an equivalence relation, and take the quotient $\tilde M$. Combining with Proposition \ref{Ozeki} and Remark \ref{remarkOzeki}, we can show that there is a bijection between $\Ext_{\rm{cris}}(\hatL', \hatL)$ and $\tilde M$, and so $\Ext_{\rm{cris}}(\hatL', \hatL)$ has an $\mathcal O_E$-module structure.

\textbf{Proof of (4)(resp. (5))} is very similar to that of (1) (resp. (2)).

\textbf{Proof of (6)}. To prove Statement(6), we use notations in the proof of Statement (1). Suppose $\tilde x \in \Ext(\hual', \hual)$ maps to $0$ in $\Ext(\barl', \barl)$, then it suffices to show that $\tilde x \in \omega_E\Ext(\hual', \hual)$.
Suppose $\tilde x$ corresponds to
$\left(
 \begin{array}{ccccc}
    A  & C \\
0 & A'
 \end{array}
\right)$, then
$\left(
 \begin{array}{ccccc}
    \overline A  & \overline C \\
0 & \overline A'
 \end{array}
\right)$ is equivalent to the trivial extension in $\Ext(\barl', \barl)$.
So there exists $\overline W \in \Mat(\huaS_{k_E})$ such that $\overline C=\overline W \overline A' -\overline A \varphi(\overline W)$. Take any lift $W \in \Mat(\huaS_{\mathcal O_E})$ of $\overline W$. Then we have $C-WA'-A\varphi(W) = \omega_E P$ for some $P \in \Mat(\huaS_{\mathcal O_E})$. So $\left(
 \begin{array}{ccccc}
    A  & C \\
0 & A'
 \end{array}
\right)$ is in fact equivalent to $\left(
 \begin{array}{ccccc}
    A  & \omega_E P \\
0 & A'
 \end{array}
\right)$.
So now it suffices to show that $\left(
 \begin{array}{ccccc}
    A  &  P \\
0 & A'
 \end{array}
\right)$ is an element in $\Ext(\hual', \hual)$. Suppose
$\left(
 \begin{array}{ccccc}
    A  & \omega_E P \\
0 & A'
 \end{array}
\right)
\left(
 \begin{array}{ccccc}
    B  & D \\
0 & B'
 \end{array}
\right)
= (E(u))^r Id$, then $AD+ \omega_E PB'=0$. We have $\omega_E \mid AD$, so $\omega_E \mid A'AD=(E(u))^r D$. Thus $\omega_E \mid D$ because $\omega_E \mid E(u)x$ in $\huaS_{\mathcal O_E}$ if and only if $\omega_E \mid x$.
So now, we have
$\left(
 \begin{array}{ccccc}
    A  &  P \\
0 & A'
 \end{array}
\right)
\left(
 \begin{array}{ccccc}
    B  & D/\omega_E \\
0 & B'
 \end{array}
\right)
=(E(u))^r Id.
$
And so $\tilde x \in \omega_E \Ext(\hual', \hual)$.

\textbf{Proof of (7)} is similar to that of Statement (6).
\end{proof}

\begin{rem} \label{invertpext}
\begin{enumerate}
  \item Let $\huaS_E=\huaS\otimes_{\Zp} E$, then we can define a category $\Mod_{\huaS_{E}}^{\varphi}$ similarly as $\Mod_{\huaS_{\mathcal O_E}}^{\varphi}$. It is clear that if $\huaM \in \Mod_{\huaS_{\mathcal O_E}}^{\varphi}$, then $\huaM [\frac{1}{p}] \in \Mod_{\huaS_{E}}^{\varphi}$.
  \item Given two modules $\huaL', \huaL \in \Mod_{\huaS_{E}}^{\varphi}$, we can define the set $\Ext(\huaL', \huaL)$ similarly as in the category $\Mod_{\huaS_{\mathcal O_E}}^{\varphi}$, and we can similarly show that $\Ext(\huaL', \huaL)$ is an $E$-vector space.
  \item We can also similarly define a category $\Mod_{\huaS_{E}}^{\varphi, \Ghat}$ and $\Ext$ in it.
\end{enumerate}
\end{rem}

The following two lemmas (Lemma \ref{81}, \ref{taushape}) are extracted from the proof of \cite[Lem. 8.1]{GLS14}.

\begin{lemma} \label{81}
Let $\zeta \in R\otimes_{\Fp}k_E$, and write it as $\zeta= \sum_{i=1}^n y_i\otimes a_i$ where $y_i \in R$, and $a_i \in k_E$ are independent over $\Fp$. Let $$v_R(\zeta):=\textnormal{min}\{v_R(y_i)\}.$$
Then $v_R$ is a well-defined valuation on $R\otimes_{\Fp}k_E $ (so in particular, it does not depend on the sum representing $\zeta$).
\end{lemma}

\begin{lemma} \label{taushape}
With notations in \textbf{(CRYS)}, and suppose $\barrho$ is upper triangular. Then $\barhatm$ is upper triangular, and there exists a basis $\{e_{s, i}\}$ for $\barhatm$, such that $\tau(\bolde_s)=\bolde_s Z_s$, and for each $s$, the matrix $Z_s$ satisfy:
\begin{itemize}
  \item $Z_s =(z_{s, i, j}) \in \Mat(\Rhat/p\Rhat\otimes_{\Fp}k_E) \subset \Mat(R\otimes_{\Fp}k_E)$ is upper triangular.
  \item The diagonal elements satisfy $v_R(z_{s, i, i}-1) \geq \frac{p^2}{p-1}, \forall i$.
  \item The elements on the upper right corner satisfy $v_R(z_{s, i, j}) \geq \frac{p^2}{p-1}, \forall i<j$.
\end{itemize}
\end{lemma}

\begin{defn} \label{shapeext}
Suppose $\barn_i, \barn_j'$ (resp. $\barhatn_i, \barhatn_j$) are rank 1 modules in $\Mod_{\huaS_{k_E}}^{\varphi}$ (resp. $\Mod_{\huaS_{k_E}}^{\varphi, \Ghat}$) for $1\leq i \leq d, 1 \leq j \leq d'$.
\begin{enumerate}
\item Let $\mathcal E_{\varphi-\rm shape}(\barn_d, \ldots, \barn_1) \subset \mathcal E(\barn_d, \ldots, \barn_1)$ be the subset consisting of elements $\barm$ such that there exists a basis $\bolde_s$ of $\barm_s$, $\varphi(\bolde_{s-1})=\bolde_sA_s$, and $A_s$ is of the shape in Proposition \ref{shape} for each $s$.

\item Suppose $\barm \in \mathcal E_{\varphi-\rm shape}(\barn_d, \ldots, \barn_1), \overline{\huaM'} \in \mathcal E_{\varphi-\rm shape}(\barn_{d'}', \ldots, \barn_{1}').$
      Define
      $\Ext_{\varphi-\rm shape}(\barm, \barm') \subseteq \Ext(\barm, \barm')$, where $\tilde x \in \Ext(\barm, \barm')$ is in $\Ext_{\varphi-\rm shape}(\barm, \barm')$ if there exists a representative of $\tilde x$:
       $0 \to\barm'  \to \barn \to \barm \to 0,$ such that $\overline{\huaN} \in \mathcal E_{\varphi-\rm shape}(\barn_d, \ldots, \barn_1, \barn_{d'}', \ldots, \barn_1').$

\item  Let
$\mathcal E_{\vtshape}(\barhatn_d, \ldots, \barhatn_1) \subset \mathcal E(\barhatn_d, \ldots, \barhatn_1)$
 be the subset of consisting of elements $\barhatm$ such that there exists a basis $\bolde_s$ of $\barm_s$ such that
  \begin{itemize}
    \item $\varphi(\bolde_{s-1})=\bolde_s A_s$ where $A_s, \forall s$ is of the shape in Proposition \ref{shape}.
    \item $\tau(1\otimes_{\varphi}\bolde_{s})=(1\otimes_{\varphi}\bolde_{s}) Z_s$, where $Z_s, \forall s$ is of the shape in Lemma \ref{taushape}.
  \end{itemize}

    \item  Define $\Ext_{\vtshape}(\barhatm, \barhatm')$ similarly as (2).
\end{enumerate}
\end{defn}

\begin{prop} \label{shapesubmodule}
With notations in Definition \ref{shapeext}, we have the following.
  \begin{enumerate}
   \item $\Ext_{\vshape}(\barm, \barm')$ is a sub-vector space of $\Ext(\barm, \barm')$.
   \item  $\Ext_{\vtshape}(\barhatm, \barhatm')$ is a sub-vector space of $\Ext(\barhatm, \barhatm').$
  \end{enumerate}
\end{prop}

\begin{proof}
For (1), from the proof of Proposition \ref{modulestr}, $\Ext(\barm, \barm')$ is bijective with some vector space $\tilde M$ (the definition of $\tilde M$ is obvious, which we omit), and $\Ext_{\varphi-\rm shape}(\barm, \barm')$ correspond to the subset of ${\tilde M}$ consisting of elements $\tilde x \in \tilde M$ which has a representative of the shape in Proposition \ref{shape}, and these elements clearly form a sub-vector space. Note that the existence of the representative is not necessarily unique, but it does not affect our result.
The proof of (2) is similar (also note that $v_R$ is a valuation).
\end{proof}

\begin{rem}
As we mentioned earlier in Remark \ref{rem: upper bound of shape}, our Proposition \ref{shapesubmodule} will give an upper bound for the shape of upper triangular Kisin modules with $k_E$-coefficient that we are studying. Indeed, the $k_E$-dimension of $\Ext_{\vshape}(\barm, \barm')$ will be precisely equal to $d_{\cris}(\barm, \barm')$ as in Definition \ref{dimNek}. See also the arguments in Theorem \ref{lifting-1}.
\end{rem}

\subsection{Extension of representations}

Now we consider extension of representations.
Let $H$ be a topological group, and let $A$ be a topological ring with trivial $H$-action. Let $L_1, L_2$ be two finite free $A$-modules with continuous $A$-linear $H$-action. Define $\Ext_H(L_2, L_1)$ to be the set of short exact sequences of finite free continuous $H$-representations over $A$, $0 \to L_1 \to N \to L_2 \to 0,$ modulo the obvious equivalence relations.
Then $\Ext_H(L_2, L_1)$ is in bijection with the continuous group cohomology $H^1 (H, \Hom_A(L_2, L_1) )$, which has a natural $A$-module structure because $\Hom_A(L_2, L_1)$ is an $A$-module.

Now let $T$ be an $\mathcal O_E$-representation of the Galois group $G_K$. Let $V= T\otimes_{\mathcal O_E} E$, and $\overline{T}=T/\omega_ET$. We have the natural $\mathcal O_E$-linear map
$\eta:H^1(G_K, T) \to H^1(G_K, V)$, and $\theta: H^1(G_K, T) \to H^1(G_K, \overline{T}).$ In the following, we list a few easy facts.

\begin{lemma}\hfill
\begin{enumerate}
\item The map $H^1(G_K, T)/\omega_EH^1(G_K, T) \to H^1(G_K, \overline T)$ is injective.
  \item The kernel of $\eta:H^1(G_K, T) \to H^1(G_K, V)$ is $H^1(G_K, T)_{\rm{tor}}$ which is the submodule consisting of elements killed by a power of $\omega_E$.

  \item $H^1(G_K, T)\otimes_{\mathcal O_E} E \simeq (H^1(G_K, T)/H^1(G_K, T)_{\rm tor})\otimes_{\mathcal O_E} E \simeq H^1(G_K, V).$

\end{enumerate}
\end{lemma}

\begin{defn}
\begin{enumerate}
  \item Suppose $V$ is a crystalline representation, then let
  $$H^1_{f}(G_K, V) : = \Ker (H^1(G_K, V) \to H^1(G_K, V\otimes_{\Qp}B_{\rm{cris}})),$$
      which is an $E$-vector space that classifies crystalline extension classes (see e.g. \cite[\S.1.12]{Nek93}).

  \item Suppose $V$ is a semi-stable representation, then let
  $$H^1_{g}(G_K, V) : = \Ker (H^1(G_K, V) \to H^1(G_K, V\otimes_{\Qp}B_{\rm{st}})),$$
      which is an $E$-vector space that classifies semistable extension classes (see e.g. \cite[\S. 1.12]{Nek93}).

  \item Let $H^1_{f}(G_K, T): = \eta^{-1}(H^1_{f}(G_K, V))$, which is an $\mathcal O_E$-module.

     \item Let $H^1_{g}(G_K, T): = \eta^{-1}(H^1_{g}(G_K, V))$, which is an $\mathcal O_E$-module.

  \end{enumerate}
\end{defn}

By the above definition, we can define
\begin{defn}
\begin{enumerate}
\item Let $V_1, V_2$ be two crystalline $E$-representations of $G_K$, then define
$$\Ext_{\rm{cris}}(V_2, V_1): = H_f^1 (G_K, \Hom_E(V_2, V_1)).$$

\item Let $T_1, T_2$ be two crystalline $\mathcal O_E$-representations of $G_K$, then define
$$\Ext_{\rm{cris}}(T_2, T_1): = H_f^1 (G_K, \Hom_{\mathcal O_E}(T_2, T_1)).$$

  \item Define $\Ext_{\rm{st}}(V_2, V_1)$ and $\Ext_{\rm{st}}(T_2, T_1)$ similarly.

\end{enumerate}
\end{defn}

\subsection{From extension of modules to extension of representations.}
In this subsection, we establish the relation between the extensions studied in the previous two subsections.

\begin{prop}[\cite{GLS14+}] \label{exttorep}
\begin{enumerate}
   \item For $\huaL', \huaL \in \Mod_{\huaS_{\mathcal O_E}}^{\varphi}$, $\Ext (\huaL', \huaL) \to \Ext_{G_{\infty}}( T_{\huaS}(\huaL), T_{\huaS}(\huaL')  )$ is a homomorphism of $\mathcal O_E$-modules.

  \item For $\hat\huaL', \hat\huaL \in \Mod_{\huaS_{\mathcal O_E}}^{\varphi, \Ghat}$, $\Ext (\hat\huaL', \hat\huaL) \to \Ext_{G_K}( \hat T(\hat\huaL), \hat T(\hat\huaL')  )$ is a homomorphism of $\mathcal O_E$-modules.

  \item Suppose $p>2$, and both $\hatL', \hatL \in \Mod_{\huaS_{\mathcal O_E}}^{\varphi, \Ghat}$ are crystalline, then $\Ext_{\rm{cris}} (\hat\huaL', \hat\huaL) \to \Ext_{\rm{cris}}( \hat T_{\huaS}(\hat\huaL), \hat T_{\huaS}(\hat\huaL')  )$ is a homomorphism of $\mathcal O_E$-modules.

  \item For $\overline{\huaL'}, \overline{\huaL} \in \Mod_{\huaS_{k_E}}^{\varphi}$,
$\Ext (\overline{\huaL'}, \overline{\huaL}) \to \Ext_{G_{\infty}}( T_{\huaS}(\overline{\huaL}), T_{\huaS}(\overline{\huaL}')  )$ is a homomorphism of $k_E$-vector spaces.

   \item For $\overline{\hat\huaL'}, \overline{\hat\huaL} \in \Mod_{\huaS_{k_E}}^{\varphi, \Ghat}$, $\Ext (\overline{\hat\huaL'}, \overline{\hat\huaL}) \to \Ext_{G_K}(\hat T(\overline{\hat\huaL}), \hat T(\overline{\hat\huaL}')  )$ is a homomorphism of $k_E$-vector spaces.

\end{enumerate}
\end{prop}
\begin{proof}
We only prove (1), the other statements can be proved similarly. We will freely use notations in the proof of Statement (1) of Proposition \ref{modulestr}. Let $e, e'$ be a fixed basis of $\huaL, \huaL'$ respectively. Suppose $\tilde x \in \Ext (\huaL', \huaL)$, take a representative of short exact sequence, $0 \to \huaL \to \huaN \to \huaL' \to 0,$
and take a representative
$   \left(
 \begin{array}{ccccc}
    A  & C \\
0 & A'
 \end{array}
\right) \in \tilde M,$
which corresponds to a section $s_C: \huaL' \to \huaN$.
We also get the corresponding short exact sequence
$T_{\huaS}(\tilde x): \quad 0 \to L' \to N \to L \to 0.$

Let $h \in L$, and set up the following matrix equation in $W(R)$, with $\alpha_{h, C}$ a row of indeterminates (in total $d'$ indeterminates, where $d'=\textnormal{rk}(\huaL')$):
$$\varphi(\alpha_{h, C}) = h(e) C + \alpha_{h, C} A'.$$
The equation is always solvable by \cite[Lem. 2.7]{Car13}, although the solution is not necessarily unique. However, since $L$ is finite free over $\mathcal O_E$, we can always fix a system of solutions which is linear with respect to $h$, i.e.,
$$\alpha_{ah_1+bh_2, C}= a\alpha_{h_1, C}+b\alpha_{h_2, C}, \forall a, b \in \mathcal O_E, h_1, h_2 \in L.$$

Now define an element $F(h, C) \in N$ such that

\[
F(h, C): \left\{
  \begin{array}{l l}
   e \mapsto h(e)\\
   s_C(e') \mapsto \alpha_{h, C}
  \end{array} \right.
  \]

The map $h \mapsto F(h, C)$ defines a section $L \to N$. Now define $c: G_{\infty} \to \Hom_{\mathcal O_E}(L, L')$ by
$$g \mapsto \{  h \mapsto  g\circ[F(h, C)]-F(g\circ h, C) \},$$
where $h \in L$, $g\circ $ is the Galois action (on $L$ or $N$). It is precisely the cocycle associated to the extension $0 \to L' \to N \to L \to 0$.

Now, take two extensions $\tilde x_1, \tilde x_2 \in \Ext(\hual', \hual)$, which correspond to $\left(
 \begin{array}{ccccc}
    A  & C_1 \\
0 & A'
 \end{array}
\right)$ and
$   \left(
 \begin{array}{ccccc}
    A  & C_2 \\
0 & A'
 \end{array}
\right)$ respectively. As above, we can fix $\alpha_{h, C_i}$ which are each linear with respect to $h$. And now define
$$\alpha_{h, aC_1+bC_2}:=  a\alpha_{h, C_1}+b\alpha_{h, C_2}.$$

We need to verify that the cocycle $c$ is ``additive" with respect to $C$, that is

\begin{multline*}
ag\circ[F(h, C_1)]-aF(g\circ h, C_1)+bg\circ[F(h, C_2)]-bF(g\circ h, C_2)\\
=g\circ[F(h, aC_1+bC_2)]-F(g\circ h, aC_1+bC_2).
\end{multline*}

Since both sides of the above formula are in $L'$, it suffices to verify their values on liftings of $e'$.

$F(g\circ h, aC_1+bC_2)$ is linear by our definition of $\alpha_{h, aC_1+bC_2}$. That is,
$$ F(g\circ h, aC_1+bC_2) |_{s_{aC_1+bC_2}(e')} = aF(g\circ h, C_1)|_{s_{C_1}(e')} + bF(g\circ h, C_2)|_{s_{C_2}(e')}.
$$

To verify on $\{g\circ[F(h, aC_1+bC_2)]\}|_{s_{aC_1+bC_2}(e')}$, just note that $L'$ is an $G_{\infty}$-invariant subspace in $N$, and so
$$\{g\circ[F(h, aC_1+bC_2)]\}|_{s_{aC_1+bC_2}(e')} = g\circ \{[F(h, aC_1+bC_2)] |_{s_{aC_1+bC_2}(e')}\}.
$$
\end{proof}

In order to prove our next proposition, we need to briefly recall some notations.

Let $S_{K_0}:= S\otimes_{W(k)}K_0$ and let $\Fil^i S_{K_0}:= \Fil^i S\otimes_{W(k)}K_0$.
Let $\bigMF$ be the category whose objects are finite free $S_{K_0}$-modules $\D$ with:
\begin{enumerate}
 \item a $\varphi_{S_{K_0}}$-semi-linear morphism $\varphi_{\D}: \D \to \D$ such that the determinant of $\varphi_{\D}$ is invertible in $S_{K_0}$;
 \item a decreasing filtration $\{\Fil^i\D\}_{i \in \Z}$ of $S_{K_0}$-submodules of $\D$ such that $\Fil^0\D=\D$ and $\Fil^i S_{K_0} \Fil^j \D \subseteq \Fil^{i+j}\D$;
 \item a $K_0$-linear map $N: \D \to \D$ such that $N(fm)=N(f)m+fN(m)$ for all $f\in S_{K_0}$ and $m \in \D$, $N\varphi=p \varphi N$ and $N (\Fil^i \D) \subseteq \Fil^{i-1}\D$.
\end{enumerate}
Morphisms in the category are $S_{K_0}$-linear maps preserving filtrations and commuting with $\varphi$ and $N$. And we can naturally define short exact sequences in the category.

We denote $\MF$ the category of filtered $(\varphi, N)$-modules, and $\MFwa$ the subcategory of weakly admissible modules. The definitions of these categories are omitted, and can be found, e.g., in \cite[\S 6.1]{Bre97}.

\begin{thm}[{ \cite[Thm 6.1.1]{Bre97}, \cite[Cor. 3.2.3]{Liu08}   }]  \label{bigMF}\hfill
\begin{enumerate}
  \item There is a functor:
  $\D: \MF \to \bigMF,$
  which is an equivalence of categories. And the equivalence and its inverse are both exact.
  \item Let $\bigMFwa$ denote the essential image of the functor $\D$ restricted to $\MFwa$, then $\D$ induces an equivalence of categories:
      $$\D: \MFwa \to \bigMFwa,$$
      and the equivalence and its inverse are both exact.
  \item With notations in Statement (3) of Theorem \ref{huaMcoeff}, suppose $D$ is the filtered $(\varphi, N)$-module associated to $V$, then there is a canonical isomorphism
   $$S_{K_0}\otimes_{\varphi, \huaS}\huaM \simeq S_{K_0}\otimes_{K_0}D \simeq \D(D),$$
   which is compatible with $\varphi, N$ and filtrations on both sides (we omit the definitions of these filtrations).
\end{enumerate}
\end{thm}

\begin{prop}[\cite{GLS14+}] \label{tensorE}
Suppose $p>2$. Let $\hat \huaL, \hat \huaL' \in \Mod_{\huaS_{\mathcal O_E}}^{\varphi, \Ghat}$, $L=\hat T(\hat \huaL), L'=\hat T(\hat \huaL') $, and $V=L \otimes_{\mathcal O_E} E, V'=L' \otimes_{\mathcal O_E} E$.
Then the following hold.
\begin{enumerate}
  \item We have the following isomorphisms of $E$-vector spaces:
  $$E\otimes_{\mathcal O_E} \Ext(\hat \huaL', \hat \huaL) \simeq E\otimes_{\mathcal O_E}\Ext_{\rm{st}}(L, L') \simeq\Ext_{\rm{st}}(V, V').$$

  \item If both $V, V'$ are furthermore crystalline, then we have
   $$E\otimes_{\mathcal O_E} \Ext_{\rm{cris}}(\hat \huaL', \hat \huaL) \simeq E\otimes_{\mathcal O_E}\Ext_{\rm{cris}}(L, L') \simeq\Ext_{\rm{cris}}(V, V').$$

\end{enumerate}
\end{prop}

\begin{proof}
Clearly we have
$E\otimes_{\mathcal O_E}\Ext(L, L') =\Ext(V, V').$
Now $\Ext_{\rm{st}}(L, L')$ is the preimage of $\Ext_{\rm{st}}(V, V')$ in the map $\Ext(L, L') \to \Ext(V, V')$, so we have $E\otimes_{\mathcal O_E}\Ext_{\rm{st}}(L, L') =\Ext_{\rm{st}}(V, V').$

Since $\hat T$ is fully faithful, it is easy to see that the map $\Ext(\hat \huaL', \hat \huaL) \to \Ext_{\rm{st}}(L, L')$ is injective.
So, in order to prove $E\otimes_{\mathcal O_E} \Ext(\hat \huaL', \hat \huaL) = E\otimes_{\mathcal O_E}\Ext_{\rm{st}}(L, L')$, it suffices to show that for any $x \in \Ext_{\rm{st}}(L, L')$, there exists some nonnegative integer $m$, such that $p^m x \in \hat{T}(\Ext(\hat \huaL', \hat \huaL))$.
Take a representative in $x$, $0 \to L' \to N \to L \to 0.$
By full faithfulness of $\hat T$, we will have a sequence
$$\hat x: \quad 0 \to \hat\huaL \to \hat\huaN \to \hat\huaL' \to 0.$$

The sequence $\hat x$ is left exact by \cite[Lem. 2.19]{Liu12}, but it is not necessarily exact (i.e., $\hat x$ is not necessarily in $\Ext(\hat \huaL' ,\hat \huaL)$), because the inverse of $\hat T$ is not exact.
However, we claim that the following sequence is short exact (note that the following sequence is by tensoring over $\huaS$, i.e., we are treating the modules $\huaL, \huaN, \huaL'$ in $\Mod_{\huaS_{\mathcal O_E}}^{\varphi, \Ghat}$ as modules in $\Mod_{\huaS_{\Zp}}^{\varphi, \Ghat}$, but it does not affect our result):
$$0 \to S[1/p]\otimes_{\huaS, \varphi}\huaL \to S[1/p]\otimes_{\huaS, \varphi}\huaN \to S[1/p]\otimes_{\huaS, \varphi}\huaL' \to 0.$$
This is because of the exact equivalences in Theorem \ref{bigMF}, and the above sequence corresponds to the short exact sequence of semi-stable representations
$$0 \to L'\otimes_{\mathcal O_E}E \to N\otimes_{\mathcal O_E}E \to L\otimes_{\mathcal O_E}E \to 0.$$

Now, since
$\huaS[1/p] \cap (S[1/p])^{\times} =(\huaS[1/p])^{\times},$
one can easily deduce that
$$\hat x[1/p]: \quad 0 \to \huaL[1/p]  \to \huaN[1/p] \to \huaL'[1/p] \to 0$$
is short exact.
The above short exact sequence is an element in $\Ext(\hat\huaL'[1/p], \hat\huaL[1/p] )$ (see Remark \ref{invertpext}).
Note that we have $p>2$, so by an analogue of Lemma \ref{pnot2}, this element in $\Ext(\hat\huaL'[1/p], \hat\huaL[1/p] )$ is determined by two matrices
$M_{\varphi} \in \Mat(\huaS\otimes_{\Zp}E) \textnormal{ and } M_{\tau} \in \Mat(\hat{R}\otimes_{\Zp}E).$
Both $M_{\varphi}$ and $M_{\tau}$ are block upper-triangular. So for $m$ big enough, $p^m$ times the upper right corner of $M_{\varphi}$ will fall in $\huaS_{\mathcal O_E}$, and $p^m$ times the upper right corner of $M_{\tau}$ will fall in $\hat{R}_{\mathcal O_E}$. That is to say, $p^m \hat x \in \Ext(\hatL', \hatL)$. And we are done for the proof of Statement (1).

The proof of Statement (2) (the crystalline case) is similar to Statement (1).
\end{proof}

\begin{rem}
By Proposition \ref{tensorE}, the $\mathcal O_E$-free parts of both $\Ext(\hat \huaL', \hat \huaL)$ and $\Ext_{\cris}(\hat \huaL', \hat \huaL)$ are finitely generated, because the cohomology groups $H_g^1$ and $H_f^1$ are finite free over $E$ (see \cite[Prop. 1.24]{Nek93}). However, we do not know much about the torsion parts of these $\mathcal O_E$-modules. Let us mention that in \cite{GLS14+}, they have constructed an (explicit) example where $\Ext_{\cris}(\hat \huaL', \hat \huaL)$ has torsion. We also have very little knowledge about the $\mathcal O_E$-module structures of those other modules as in Prop. \ref{modulestr}.
\end{rem} 

\section{Two conditions on upper triangular extensions}

In this section, we prove two useful propositions, which will be used in our crystalline lifting theorems. The first one is about restricting group cohomology from $G_K$ to $G_{\infty}$, the second one is about equip $\tau$-actions to modules in $\Mod_{\huaS_{k_E}}^{\varphi}$. The results in this section are valid for any finite extension $K/\Qp$ with $p>2$.

\begin{prop} \label{injectivity}
Let $\barzeta_i : G_K \to k_E^{\times}, 1\leq i \leq d$ be characters, such that $\barzeta_i \neq \mathbbm{1}$ or $\overline \varepsilon_p$, where $\mathbbm{1}$ is the trivial character and $\overline \varepsilon_p$ is the reduction of the cyclotomic character.
Suppose $W \in \mathcal E_{G_K}(\barzeta_1, \ldots, \barzeta_d)$, then the restriction map
$H^1(G_K, W) \to H^1(G_{\infty}, W)$
is injective.
\end{prop}

\begin{proof}
The proof imitates that of \cite[Lem. 7.4.3]{EGS14}. We write it out in more detail, for the reader's convenience. Recall that $\hat K = K_{p^{\infty}, \infty}$, $G_{\hat K}=\Gal(\overline K/\hat K)$, $\hat G=\Gal(\hat K/K)$ and $G_{p^{\infty}}=\Gal(\hat K/K_{p^{\infty}})$. To prove that $H^1(G_K, W) \to H^1(G_{\infty}, W)$ is injective, it suffices to show that the composite
$H^1(G_K, W) \to H^1(G_{\infty}, W) \to H^1(G_{\hat K}, W)$
is injective. By inflation-restriction, it suffices to show that $H^1(\hat G, W^{G_{\hat K}})=0$.
Denote $W_1=W^{G_{\hat K}}$, by inflation-restriction, we have
$$ 0 \to H^1(\Gal(K_{p^{\infty}}/K), W_2) \to H^1(\hat G, W_1) \to H^1(G_{p^{\infty}}, W_1)^{\Gal(K_{p^{\infty}}/K)}.$$

Where $W_2= W_1^{\Gal(\hat K/K_{p^{\infty}})}= W^{\Gal(\overline K/K_{p^{\infty}})}$. Again by inflation-restriction,
\begin{multline*}
0 \to H^1(\Gal(K(\mu_p)/K), W_3) \to H^1(\Gal(K_{p^{\infty}}/K), W_2) \\
\to H^1(\Gal(K_{p^{\infty}}/K(\mu_p)), W_2)^{\Gal(K(\mu_p)/K)}.
\end{multline*}

Where $W_3 = W_2^{\Gal(K_{p^{\infty}}/K(\mu_p))}=W^{\Gal(\overline K/K(\mu_p))}$. So now, it suffices to prove that
$ H^1(\Gal(K_{p^{\infty}}/K(\mu_p)), W_2)^{\Gal(K(\mu_p)/K)}=0$, $H^1(\Gal(K(\mu_p)/K), W_3)=0$ and $H^1(G_{p^{\infty}}, W_1)^{\Gal(K_{p^{\infty}}/K)}=0$.

(1). To show that $H^1(\Gal(K_{p^{\infty}}/K(\mu_p)), W_2)^{\Gal(K(\mu_p)/K)}=0$, note that\\
 $\Gal(K_{p^{\infty}}/K(\mu_p)) \simeq \Zp$ is abelian, so the action of $\Gal(K_{p^{\infty}}/K(\mu_p))$ on $W_2$ is via a sum of characters.
For $\theta: \Zp \to k_E^{\times}$ a character, $H^1(\Zp, \theta) =0$ unless $\theta$ is the trivial character. So $H^1(\Gal(K_{p^{\infty}}/K(\mu_p)), W_2)^{\Gal(K(\mu_p)/K)} \neq 0$ only if there exists $0 \neq v \in W^{G_K}$, which is impossible.

(2). It is easy to show that $H^1(\Gal(K(\mu_p)/K), W_3)=0$, using that $\Gal(K(\mu_p)/K)$ is a subgroup of $\Fp^{\times}$.

(3). To show $H^1(G_{p^{\infty}} , W_1)^{\Gal(K_{p^{\infty}}/K)}=0$, note that $G_{p^{\infty}} \simeq \Zp$ is abelian. For $\theta: \Zp \to k_E^{\times}$ a character, $H^1(\Zp, \theta) \neq 0$ only if $\theta$ is the trivial character. So if $H^1(G_{p^{\infty}} , W_1)^{\Gal(K_{p^{\infty}}/K)} \neq 0$, there exists $0 \neq v \in W_1$ which generates a trivial character of $G_{p^{\infty}}$, i.e., $e \in W^{\Gal(\overline K/K_{p^{\infty}})}$, and such that $\Gal(K_{p^{\infty}}/K)$ acts on $k_E\cdot e$ via the cyclotomic character, i.e., $\overline \varepsilon_p \subset W$, contradiction.
\end{proof}

\noindent \textnormal{\textbf{{Condition (C-2A)}}} Let $\barzeta_i : G_K \to k_E^{\times}, 1\leq i \leq d$ be characters. We say that the ordered sequence $(\barzeta_1, \ldots, \barzeta_d)$ satisfies Condition \textbf{(C-2A)} if $\barzeta_i^{-1} \barzeta_j \neq \mathbbm{1}$ or $\overline{\varepsilon}_p$, for all $i <j$. When $(\barzeta_1, \ldots, \barzeta_d)$ satisfies \textbf{(C-2A)}, for any $\barrho \in \mathcal E(\barzeta_1, \ldots, \barzeta_d)$, we also say $\barrho$ satisfies \textbf{(C-2A)}. By Proposition \ref{injectivity}, when $(\barzeta_1, \ldots, \barzeta_d)$ satisfies \textbf{(C-2A)}, then for any $1 \leq i <d$, and any $W \in \mathcal E(\barzeta_i^{-1}\barzeta_{i+1}, \ldots, \barzeta_i^{-1}\barzeta_d)$, the map $H^1(G_K, W) \to H^1(G_{\infty}, W)$ is injective.

\begin{rem}
  \begin{enumerate}
\item When $d=2$, our Condition \textnormal{\textbf{{ (C-2A)}}} is related to peu ramifi\'{e} and tr\`{e}s ramifi\'{e} extensions of characters as in \cite[\S 9]{GLS14}. When $d>2$, things get even more complicated, see e.g., the discussions in \cite{GHLS}.

    \item After the current paper was posted, \cite[Lem. 3.9]{LLHLM} proved a more general statement than our Proposition \ref{injectivity}. In particular, using \cite[Lem. 3.9]{LLHLM}, we can weaken the Condition \textnormal{\textbf{{ (C-2A)}}} such that we only need to require $\barzeta_i^{-1} \barzeta_j \neq \overline{\varepsilon}_p$, (and thus slightly improve our main theorems).
  \end{enumerate}
\end{rem}

We define another condition for our next proposition, which is a direct generalization of \cite[Lem. 8.1]{GLS14}.
\smallskip

\noindent \textnormal{\textbf{{Condition (C-2B)}}} Let $(\barn_1, \ldots, \barn_d)$ be an \emph{ordered} sequence of rank-$1$ modules in $\Mod_{\huaS_{k_E}}^{\varphi}$. We say $(\barn_1, \ldots, \barn_d)$ satisfies Condition \textbf{(C-2B)}, if there does not exist $i<j$, such that
$$\barn_i=\barm(0, \ldots, 0 ;a_i),\quad \barn_j=\barm(p, \ldots, p ;a_j)$$
for some $a_i, a_j \in k_E^{\times}$.
When $(\barn_1, \ldots, \barn_d)$ satisfies \textbf{(C-2B)}, then for any $\barm \in \mathcal E(\barn_d, \ldots, \barn_1)$, we also say $\barm$ satisfies \textbf{(C-2B)}.

\smallskip

\begin{prop} \label{forgettau}

Suppose $\barhatl \in \mathcal E_{\vtshape}(\barhatn_r, \ldots, \barhatn_1)$, $\barhatl' \in \mathcal E_{\vtshape}(\barhatn_d, \ldots, \barhatn_{r+1}),$
where $\overline{\hat\huaN}_i, 1 \leq i \leq d$ are rank-$1$ modules in $\Mod_{\huaS_{k_E}}^{\varphi, \Ghat}$ such that the ordered sequence $(\barn_1, \ldots, \barn_d)$ satisfies Condition \textbf{(C-2B)}. Then the map (forgetting the $\tau$-action)
$$ \Ext_{\vtshape}(\barhatl', \barhatl) \to  \Ext_{\vshape}(\barl', \barl) $$
is a $k_E$-linear homomorphism, and it is injective.
\end{prop}
\begin{proof}
The map is clearly $k_E$-linear. To prove injectivity, one can easily check that it reduces to the following lemma.
\end{proof}

\begin{lemma}
Let $\barm \in \mathcal E_{\varphi- \rm shape}(\barn_d, \ldots, \barn_1)$, such that the ordered sequence $(\barn_1, \ldots, \barn_d)$ satisfies Condition \textbf{(C-2B)}. Let
\begin{itemize}
  \item $A_s \in \Mat(k_E\llb u\rrb), 0 \leq s \leq f-1$ be a set of matrices for $\varphi_{\barm_s}$ of the shape in Proposition \ref{shape}, and
  \item $Z_s, Z_s' \in \Mat(R\otimes_{\Fp}k_E)$ be a set of matrices of the shape in Lemma \ref{taushape}.
\end{itemize}
If we have $Z_{s+1} \tau(\varphi(A_s)) =\varphi(A_s)\varphi(Z_s), \forall s$, and $Z_{s+1}' \tau(\varphi(A_s)) =\varphi(A_s)\varphi(Z_s'), \forall s$. Then $Z_s= Z_s', \forall s$.
\end{lemma}
\begin{proof}
This is easy generalization of \cite[Lem. 8.1]{GLS14}, and we sketch the proof.
Just as in the proof of \textit{loc. cit.}, we can expand out the matrix equation
$ Z_{s+1} \tau(\varphi(A_s)) =\varphi(A_s)\varphi(Z_s),$
and then one can easily see that the diagonal elements $Z_{s, i, i}$ are uniquely determined.
Then, one can argue as in the same fashion as in \textit{loc. cit.}, that $z_{s, i, i+1}$ are uniquely determined (so long that Condition \textbf{(C-2B)} is satisfied). Then similarly, we can show that $z_{s, i, i+2}$ are uniquely determined, and so on for all elements in $Z_s$.
\end{proof}

\section{Main local results: crystalline lifting theorems}

\newcommand{\Nek}{\textnormal{cris}}

\newcommand{\rk}{\textnormal{rk}}
 
First we introduce some useful definitions.

\begin{defn} \label{dimNek}
\begin{enumerate}
  \item Suppose $\huaL, \huaL' \in \Mod_{\huaS_{\mathcal O_E}}^{\varphi}$, where $\textnormal{rk}(\huaL')=1, \textnormal{rk}(\huaL)=d-1$, and $\huaL \in \mathcal E(\huaN_d, \ldots, \huaN_2)$. Suppose $\huaL' =\huaM(t_{1, 0}, \ldots, t_{1, f-1}; a_1)$, $\huaN_i =\huaM(t_{i, 0}, \ldots, t_{i, f-1}; a_i), \forall 2\leq i \leq d$ are rank-1 modules as in Definition \ref{huaMrank1}.
      Then define
       $$d_{\Nek}(\huaL, \huaL')= \# \{(i, s) \mid 2 \leq i \leq d, 0 \leq s \leq f-1, t_{i, s} > t_{1, s}  \}.$$

   \item Define $d_{\Nek}$ for similar pairs in $\Mod_{\huaS_{k_E}}^{\varphi}$, $\Mod_{\huaS_{\mathcal O_E}}^{\varphi, \Ghat}$, $\Mod_{\huaS_{k_E}}^{\varphi, \Ghat}$ analogously.
    \item Suppose $\rho, \rho'$ crystalline $E$-representation of $G_K$, with $\dim_E \rho'=1$ and $\rho \in \mathcal E(\chi_2, \ldots, \chi_d)$.
        Suppose $\HT_s(\rho')=t_{1, s}$, $\HT_s(\chi_i)=t_{i, s}$.
         Then define
       $$d_{\Nek}(\rho', \rho)= \# \{(i, s) \mid 2 \leq i \leq d, 0 \leq s \leq f-1, t_{i, s} > t_{1, s}  \}.$$
\end{enumerate}
\end{defn}

\begin{rem} \label{remNek}
It is clear that these definition of $d_{\Nek}$ are compatible with each other, i.e., the following statements holds:
\begin{enumerate}
  \item Suppose $\huaL, \huaL' \in \Mod_{\huaS_{\mathcal O_E}}^{\varphi}$ as in Definition \ref{dimNek}(1), then $d_{\Nek}(\huaL, \huaL')=d_{\Nek}(\barl, \barl').$

  \item Suppose $\hat \huaL, \hat \huaL' \in \Mod_{\huaS_{\mathcal O_E}}^{\varphi, \Ghat}$ are crystalline, where $\rk(\hat \huaL')=1$, $\rk(\hat \huaL)=d-1$ and $\hat \huaL$ is upper triangular. Let $\rho =\hat T(\hat \huaL)\otimes_{\mathcal O_E}E, \rho' =\hat T(\hat \huaL')\otimes_{\mathcal O_E}E$ be the associated crystalline representations. Then
      $d_{\Nek}(\rho', \rho)=d_{\Nek}(\hat \huaL, \hat\huaL').$
\end{enumerate}
\end{rem}

\begin{rem} \label{remdim}
With notations in Statement (3) of Definition \ref{dimNek}, we have
$ d_{\Nek}(\rho', \rho) = \dim_E \Ext_{\textnormal{cris}}(\rho', \rho).$
See e.g., \cite[Prop. 1.24]{Nek93}.
\end{rem}

\subsection{First crystalline lifting theorem}

\begin{thm} \label{lifting-1}
Suppose $\barhatm \in \mathcal E_{\vtshape}(\barhatn_d, \ldots, \barhatn_1)$, where $\barhatn_i =\overline{\hat\huaM}(t_{i, 0}, \ldots, t_{i, f-1}; a_i)$ are rank-$1$ modules with $t_{i, s} \in [0, p], \forall 1\leq i \leq d, 0 \leq s \leq f-1$, and $t_{i, s} \neq t_{j, s} \forall i \neq j$.
Suppose the following assumptions are satisfied:
\begin{enumerate}
  \item For any $1 \leq i < j \leq d$, there is no nonzero morphism $\barn_j \to \barn_i$.
  \item $\barm$ satisfies Condition \textbf{C-2(B)}.
\end{enumerate}
Then $\overline{\hat\huaM}$ has a lift $\hat\huaM \in \Mod_{\huaS_{\mathcal O_E}}^{\varphi, \Ghat}$ which is crystalline and upper triangular.
\end{thm}

\begin{proof}
We prove by induction on $d$. The case $d=1$ is trivial from Proposition \ref{rank1}. Suppose the theorem is true for $d-1$, and we now prove it for $d$.

Suppose $\barhatm \in \Ext_{\vtshape}(\barhatm_2, \barhatm_1)$, where $\barhatm_2 \in \mathcal E_{\varphi, \tau- \textnormal{shape}}(\barhatn_{d}, \ldots, \barhatn_2)$ is of rank $d-1$, and $\barhatm_1$ is of rank 1. We denote
$d_{\Nek} := d_{\Nek}(\barm_2, \barm_1).$
Because of assumption (2), by Proposition \ref{forgettau}, we have the injective homomorphism
$$\Ext_{\vtshape} (\barhatm_2, \barhatm_1) \hookrightarrow \Ext_{\vshape}(\barm_2, \barm_1).$$
And because of assumption (1), by Proposition \ref{shapesubmodule} and the definition of $\Ext_{\vshape}$, $\Ext_{\vshape}(\barm_2, \barm_1)$ is a $k_E$-vector space of dimension at most $d_{\Nek}$. So we have that $\dim_{k_E} \Ext_{\vtshape} (\barhatm_2, \barhatm_1) \leq d_{\cris}$.

By the induction hypothesis (for $\barhatm_2$) and Proposition \ref{rank1} (for $\barhatm_1$), we can take upper triangular crystalline lifts $\hat \huaM_1, \hat\huaM_2$ of $\overline{\hat\huaM}_1, \overline{\hat\huaM}_2$ respectively. By Proposition \ref{exttorep}, we have the natural map:
$ \Ext_{\textnormal{cris}} (\hat\huaM_2, \hat\huaM_1) \to \Ext_{\textnormal{cris}}(\That(\hatm_1), \That(\hatm_2)),$
which becomes an isomorphism after tensoring with $E$ by Proposition \ref{tensorE}, i.e.,:
\begin{eqnarray*}
  \Ext_{\textnormal{cris}} (\hatm_2, \hatm_1)\otimes_{\mathcal O_E}E &\simeq&
\Ext_{\textnormal{cris}}(\That(\hatm_1), \That(\hatm_2))\otimes_{\mathcal O_E}E\\
  &\simeq&\Ext_{\textnormal{cris}}(\That(\hat\huaM_1)\otimes_{\mathcal O_E}E, \That(\hat\huaM_2)\otimes_{\mathcal O_E}E).
\end{eqnarray*}

By Remark \ref{remdim} and Remark \ref{remNek},
$\dim_E \Ext_{\textnormal{cris}}(\That(\hat\huaM_1)\otimes_{\mathcal O_E}E, \That(\hat\huaM_2)\otimes_{\mathcal O_E}E)=d_{\Nek}.$
So we also have $\dim_E \Ext_{\textnormal{cris}} (\hatm_2, \hatm_1)\otimes_{\mathcal O_E}E=d_{\Nek}.$
Thus the $\mathcal O_E$-free part of $\Ext_{\textnormal{cris}} (\hatm_2, \hatm_1)$ has rank equal to $d_{\Nek}$.
Now the image of the injective homomorphism
$\Ext_{\textnormal{cris}} (\hatm_2, \hatm_1)/\omega_E \hookrightarrow \Ext(\barhatm_2, \barhatm_1)$
falls into $\Ext_{\vtshape} (\barhatm_2, \barhatm_1)$ by \cite[Cor. 5.10]{GLS14}. So we have the following injective homomorphism
$$\Ext_{\textnormal{cris}} (\hatm_2, \hatm_1)/\omega_E \hookrightarrow \Ext_{\vtshape} (\barhatm_2, \barhatm_1).$$
Now, the left hand side has $k_E$-dimension at least $d_{\Nek}$, and the right hand side has $k_E$-dimension at most $d_{\Nek}$. So in fact, the above homomorphism is an isomorphism, which means that every extension in $\Ext_{\vtshape} (\barhatm_2, \barhatm_1)$ has an upper triangular crystalline lift. And we finish the proof for $d$.
\end{proof}

\subsection{Second crystalline lifting theorem}
In order to prove our second crystalline lifting theorem, we need to introduce some definitions. We could have defined them earlier, but we did not need them until now.

\begin{defn} \label{Esquare}
Let $\barn_i \in \Mod_{\huaS_{k_E}}^{\varphi}$ be rank $1$ modules. Let
  $ \mathcal E_{\varphi -\textnormal{shape}}^{\square}(\barn_d, \ldots, \barn_1) $
  be the set consisting of sequences
  $\barm^{\square} = (\barm_{1, 2}, \barm_{1, 2, 3}, \ldots, \barm_{1, \ldots, d}),$
  where
  \begin{itemize}
    \item $\barm_{1, 2} \in \Ext_{\varphi -\textnormal{shape}}(\barn_2, \barn_1)$, and
    \item inductively, $\barm_{1, 2, \ldots, i+1} \in \Ext_{\varphi -\textnormal{shape}}(\barn_{i+1}, \barm_{1, 2, \ldots, i}), \forall 2\leq i \leq d-1$.
  \end{itemize}
Denote $\barm=\barm_{1, \ldots, d} \in \mathcal E(\barn_d, \ldots, \barm_1).$
We say that $\barm^{\square} = (\barm_{1, 2}, \barm_{1, 2, 3}, \ldots, \barm_{1, \ldots, d})$ gives a ``successive $\Ext$ structure" to $\barm$.
\end{defn}

\begin{rem}\label{remEsquare}
  \begin{enumerate}
    \item There is a natural ``forgetful" map of sets:
    $$\mathcal E_{\varphi -\textnormal{shape}}^{\square}(\barn_d, \ldots, \barn_1) \to \mathcal E(\barn_d, \ldots, \barm_1),$$
    where $\barm^\square \mapsto \barm$. The map is clearly surjective.
    \item It is easy to see that when $d=2$, we have the following natural bijection
    $$\mathcal E_{\varphi -\textnormal{shape}}^{\square}(\barn_2, \barn_1) = \Ext_{\varphi -\textnormal{shape}}(\barn_2, \barn_1).$$
    In particular, $\mathcal E_{\varphi -\textnormal{shape}}^{\square}(\barn_2, \barn_1)$ is endowed with a natural $k_E$-vector space structure. But when $d \geq 3$, we no longer have similar result.
  \end{enumerate}
\end{rem}

\begin{defn} \label{extsquare} 
 Suppose
 $\barm^\square =(\barm_{1, 2}, \barm_{1, 2, 3}, \ldots, \barm_{1, \ldots, d})
 \in \mathcal E_{\varphi -\textnormal{shape}}^{\square}(\barn_d, \ldots, \barn_1)$
 as in Definition \ref{Esquare},
 $\barm' \in \Mod_{\huaS_{k_E}}^{\varphi}$ a rank $1$ module.
Let
 $\Ext^{\square}_{\varphi-\textnormal{shape}}(\barm^\square, \barm')$
 be the set consisting of
 $(\barm_{1, 2}, \ldots, \barm_{1, \ldots, d}, \barl),$
 where $\barl \in \Ext_{\varphi-\textnormal{shape}}(\barm', \barm_{1, \ldots, d})$.
\end{defn}

\begin{rem} \label{remextsquare}
\begin{enumerate}
  \item It is clear that there is a natural bijective map:
$$  \Ext^{\square}_{\varphi-\textnormal{shape}}(\barm^\square, \barm') \to \Ext_{\varphi-\textnormal{shape}}(\barm, \barm'),$$
by sending $(\barm_{1, 2}, \ldots, \barm_{1, \ldots, d}, \barl)$ to $\barl$. So in particular, $  \Ext^{\square}_{\varphi-\textnormal{shape}}(\barm^\square, \barm')$ has a $k_E$-vector space structure.
  \item In the category $\Mod_{\huaS_{k_E}}^{\varphi, \Ghat}$, we can define similar sets $\mathcal E_{\vtshape}^{\square}(\barhatn_d, \ldots, \barhatn_1)$ and $\Ext^{\square}_{\vtshape}(\barhatm^\square, \barhatm')$.
\end{enumerate}
\end{rem}

We can define similar $\square$-extensions for representations, but we need to reverse the orders, so that it will be compatible with the $\square$-extensions for modules.
\begin{defn}
Let $? =G_K$ or $G_{\infty}$, let $\chi_i$ be some ($k_E$ or $\mathcal O_E$)-characters of $?$.
\begin{enumerate}
  \item Let $\mathcal E_{?}^{\square}(\chi_1, \ldots, \chi_d)$ be the set of sequences
$\rho^{\square} = (\rho_{1, 2}, \rho_{1, 2, 3}, \ldots, \rho_{1, \ldots, d})$
where $\rho_{1, 2} \in \Ext_{?}(\chi_1, \chi_2)$, and inductively, $\rho_{1, \ldots, i+1} \in \Ext_{?}(\rho_{1, \ldots, i}, \chi_{i+1})$.

\item Suppose $\rho^{\square} = (\rho_{1, 2}, \rho_{1, 2, 3}, \ldots, \rho_{1, \ldots, d})$ as above, and $\rho'$ a character of ${?}$, then let
      $\Ext_{?}^{\square}(\rho', \rho^{\square})$ be the set consisting of
      $(\rho_{1, 2}, \rho_{1, 2, 3}, \ldots, \rho_{1, \ldots, d}, r),$
      where $r \in \Ext_{?}(\rho', \rho_{1, \ldots, d})$.
\end{enumerate}
\end{defn}

\begin{thm} \label{lifting-2}
Let $\barhatm^{\square} \in \mathcal E_{\vtshape}^{\square}(\barhatn_d, \ldots, \barhatn_1)$ where $\barhatn_i =\barhatm(t_{i, 0}, \ldots, t_{i, f-1}; a_i)$ with $t_{i, s} \in [0, p]$ such that $t_{i, s} \neq t_{j, s} \forall i \neq j$.
Let $\overline{\chi}_i = \hat{T}(\barhatn_i)$, and fix $\chi_i$ a crystalline character which lifts $\overline{\chi}_i$ and $\HT_s(\chi_i)=\{ t_{i, s} \}$.
Suppose the following assumptions are satisfied:
\begin{enumerate}
  \item $\{\barchi_1, \ldots, \barchi_d\}$ has a unique model with respect to $\{\textnormal{col}_0(\WT(\barm)), \ldots, \textnormal{col}_{f-1}(\WT(\barm))\}$, where we regard $\textnormal{col}_s(\WT(\barm))$ as an (unordered) set of numbers for all $s$.

  \item $\barrho^{\square}:=\hat T_{\huaS}(\barhatm^{\square})$ satisfies Condition \textbf{C-2(A)}.
\end{enumerate}
Then there exists a
$\rho'^{\square} \in \mathcal E^{\square}_{G_K, \textnormal{cris}}(\chi_1, \ldots, \chi_d)$
such that
\begin{itemize}
  \item $\barrho'^{\square} =\barrho^{\square}$ as elements in $\mathcal E^{\square}_{G_K}(\barchi_1, \ldots, \barchi_d)$ (in particular $\barrho' \simeq \barrho$), and
  \item $\HT_s(\rho')= \col_s(\WT(\barm))$ as sets of numbers.
\end{itemize}
\end{thm}

\begin{proof}
We say \textbf{(S1)} is true if the statement in the theorem is true. We say \textbf{(S2)} is true if the following natural map of sets is injective (note that here injective means same image implies same preimage)
\begin{equation*} \tag{Map-S2} \label{Map-S2}
T_{\huaS}^{\square}: \mathcal E_{\varphi -\textnormal{shape}}^{\square}(\barn_d, \ldots, \barn_1) \to \mathcal E_{G_{\infty}}^{\square} (\barchi_1, \ldots, \barchi_d),
\end{equation*}
where for brevity, we write $\mathcal E_{G_{\infty}}^{\square} (\barchi_1, \ldots, \barchi_d)$ to mean $\mathcal E_{G_{\infty}}^{\square} (\barchi_1\mid_{G_{\infty}}, \ldots, \barchi_d\mid_{G_{\infty}})$. We prove (S1) and (S2) at the same time by induction on $d$.
We first prove for $d=2$. Note that in this case, since $\mathcal E_{?}^\square (\ast, \ast) = \Ext_{?}(\ast, \ast)$, we actually do not need $\square$.

Suppose $\barhatm \in \mathcal E_{\vtshape} (\barhatm_2, \barhatm_1)$,
$\barrho  \in \Ext_{G_K}(\barrho_1, \barrho_2)$, and denote $\rho_1, \rho_2$ the crystalline liftings of $\barrho_1, \barrho_2$ respectively. (The change of notations from $\barhatn_i$ to $\barhatm_i$, and from $\chi_i$ to $\rho_i$ is convenient for the induction process.) Denote $d_{\textnormal{cris}}=d_{\textnormal{cris}}(\barm_2, \barm_1)$.

Consider the following composite of homomorphisms:
$$f: \Ext_{\textnormal{cris}}(\rho_1, \rho_2) \twoheadrightarrow \Ext_{\textnormal{cris}}(\rho_1, \rho_2)/\omega_E \hookrightarrow \Ext_{G_K}(\barrho_1, \barrho_2) \hookrightarrow \Ext_{G_{\infty}}(\barrho_1, \barrho_2).$$
Where the last map is injective because of assumption (2) and Proposition \ref{injectivity}.

By Proposition \ref{tensorE},
$\Ext_{\textnormal{cris}}(\rho_1, \rho_2)\otimes_{\mathcal O_E} E = \Ext_{\textnormal{cris}}(\rho_1\otimes_{\mathcal O_E} E, \rho_2\otimes_{\mathcal O_E} E).$
Since $\dim_{E} \Ext_{\textnormal{cris}}(\rho_1\otimes_{\mathcal O_E} E, \rho_2\otimes_{\mathcal O_E} E)$ is equal to $d_{\textnormal{cris}}(\rho_1, \rho_2)=d_{\textnormal{cris}}$. So the $\mathcal O_E$-free part of $\Ext_{\textnormal{cris}}(\rho_1, \rho_2)$ is of rank $d_{\textnormal{cris}}$. So we have
$$\dim_{k_E} (\textnormal{Im}(f)) \geq d_{\textnormal{cris}}.$$

Now let
$\Ext_{G_{\infty}}^{\textnormal{cris}}(\barrho_1, \barrho_2)[\WT(\barm)]$
be the subset of $\Ext_{G_{\infty}}(\barrho_1, \barrho_2)$ where
$0 \to \barrho_2 \to \bar r \to \barrho_1 \to 0$
is in $\Ext_{G_{\infty}}^{\textnormal{cris}}(\barrho_1, \barrho_2)[\WT(\barm)]$, if there exists $W$ crystalline such that
\begin{itemize}
  \item $\HT_s(W)= \col_s(\WT(\barm))$ as sets for all $s$, and
  \item $\overline{W}\mid_{G_{\infty}} \simeq \bar r$.
\end{itemize}
Remark that it is not clear from the definition if $\Ext_{G_{\infty}}^{\textnormal{cris}}(\barrho_1, \barrho_2)[\WT(\barm)]$ is a sub-vector space of $\Ext_{G_{\infty}}(\barrho_1, \barrho_2)$. We have
$\textnormal{Im}(f) \subseteq  \Ext_{G_{\infty}}^{\textnormal{cris}}(\barrho_1, \barrho_2)[\WT(\barm)],$
and so we have
\begin{equation*}\tag{lower bound} \label{lower bound}
\#(\Ext_{G_{\infty}}^{\textnormal{cris}}(\barrho_1, \barrho_2)[\WT(\barm)]) \geq |k_E|^{d_{\Nek}}.
\end{equation*}

Now consider the homomorphism
$$g : \Ext_{\varphi-\textnormal{shape}}(\barm_2, \barm_1) \to \Ext_{G_{\infty}}(\barrho_1, \barrho_2).$$
Suppose
$x_{\bar r}: \quad 0 \to \barrho_2 \to \overline r \to \barrho_1 \to 0$
is an element in $\Ext_{G_{\infty}}^{\textnormal{cris}}(\barrho_1, \barrho_2)[\WT(\barm)]$, then there exists crystalline $W$ such that $\overline W\mid_{G_{\infty}} \simeq \overline r$. Let $\hat \huaL \in \Mod_{\huaS_{\mathcal O_E}}^{\varphi, \Ghat}$ be the module associated to $W$. Then by \cite[Lem. 4.4]{Oze13}, there exists a sequence
$x_{\barl}: \quad 0 \to \barl_1 \to \barl \to \barl_2 \to 0 ,$
such that $T_{\huaS}(x_{\barl}) = x_{\overline r}$. Note that we have $\barl \in \Ext_{\varphi-\textnormal{shape}}(\barl_2, \barl_1)$ by Proposition \ref{shape}.
However, because of assumption (1), we must have $\barl_i =\barm_i, i=1, 2$, and so $\barl \in \Ext_{\varphi-\textnormal{shape}}(\barm_2, \barm_1)$.
That is to say, we must have
$\textnormal{Im}(g) \supseteq \Ext_{G_{\infty}}^{\textnormal{cris}}(\barrho_1, \barrho_2)[\WT(\barm)],$
and so,
\begin{equation*}\tag{upper bound} \label{upper bound}
\#(\Ext_{G_{\infty}}^{\textnormal{cris}}(\barrho_1, \barrho_2)[\WT(\barm)]) \leq |k_E|^{d_{\Nek}}.
\end{equation*}

Combining the lower bound and upper bound we obtained above, we must have
$$\#(\Ext_{G_{\infty}}^{\textnormal{cris}}(\barrho_1, \barrho_2)[\WT(\barm)]) = |k_E|^{d_{\Nek}},$$
and so
$$\textnormal{Im}(f)=\Ext_{G_{\infty}}^{\textnormal{cris}}(\barrho_1, \barrho_2)[\WT(\barm)]=\textnormal{Im}(g).$$

Now, we can prove (S1) for $d=2$. Since $\barrho \mid_{G_{\infty}} \in \textnormal{Im}(g)$, so $\barrho \mid_{G_{\infty}} \in \textnormal{Im}(f)$, i.e., there is an upper triangular crystalline lift $\rho' \in \Ext_{\textnormal{cris}}(\rho_1, \rho_2)$ such that $\barrho'\mid_{G_{\infty}} = \barrho$. However, we must have $\barrho'= \barrho$ as elements in $\Ext_{G_K}(\barrho_1, \barrho_2)$, because the map
$\Ext_{G_K}(\barrho_1, \barrho_2) \to \Ext_{G_{\infty}}(\barrho_1, \barrho_2)$
is injective. And this proves (S1). Note that since $g$ is injective, so the map \ref{Map-S2} is injective, and so (S2) is also true.


Now, we can use induction to proceed from $d-1$ to $d$. Suppose both (S1) and (S2) are true when the dimension is $\leq d-1$, and we now prove them when the dimension becomes $d$. Now we have
$\barhatmsq \in \Ext_{\vtshape}^{\square}(\barhatmsq_2, \barhatm_1)$, where $\barhatmsq_2  \in \mathcal E_{\vtshape}^{\square}(\barhatn_d, \ldots, \barhatn_2)$ is of rank $d-1$, and $\barhatm_1$ is of rank 1. So $\barrhosq \in \Ext_{G_K}^{\square} (\barrho_1, \barrhosq_2)$. Then we can apply the induction hypothesis to $\barrhosq_2$ to find an upper triangular crystalline lift $\rhosq_2$. And now the proof is almost verbatim as the $d=2$ case, by using the $\square$-variants of above maps, i.e., let
$$f: \Ext_{\cris}^{\square}(\rho_1, \rhosq_2) \twoheadrightarrow \Ext_{\cris}^{\square}(\rho_1, \rho_2^{\square})/\omega_E \hookrightarrow \Ext_{G_K}^{\square}(\barrho_1, \barrho_2^{\square}) \hookrightarrow \Ext_{G_{\infty}}^{\square}(\barrho_1, \barrho_2^{\square}),$$
then $\textnormal{Im}(f) \subseteq \Ext_{G_{\infty}}^{{\square}, \textnormal{cris}}(\barrho_1, \barrho_2^{\square})[\WT(\barm)].$

Then consider
$$g : \Ext_{\varphi-\textnormal{shape}}^{\square}(\barm_2^{\square}, \barm_1) \to \Ext_{G_{\infty}}^{\square}(\barrho_1, \barrho_2^{\square}).$$
Suppose $\bar r \in \Ext_{G_{\infty}}^{{\square}, \textnormal{cris}}(\barrho_1, \barrho_2^{\square})[\WT(\barm)]$, then by \cite[Lem. 4.4]{Oze13}, we know that there exists some $\barl \in \Ext_{\varphi-\textnormal{shape}}^{\square}(\barl_2^{\square}, \barl_1)$ such that $T_{\textnormal{cris}}(\barl)=\bar r$.
Using our assumption (1), we know that $\barl_1=\barm_1$ and $\barl_2^{\square} \in \mathcal E^{\square}_{\vshape}(\barn_d, \ldots, \barn_2)$. But with only \cite[Lem. 4.4]{Oze13}, we do not know if $\barl_2^{\square}$ and $\barm_2^{\square}$ are the same! The \textbf{key point} here is that we can apply our induction hypothesis on the statement (S2) to conclude $\barl_2^{\square}=\barm_2^{\square}$! And this is the whole point that we have introduced all these $\square$-notations (when $d=2$, we do not need them).

So with above argument, we still have
$\textnormal{Im}(g) \supseteq \Ext_{G_{\infty}}^{{\square}, \textnormal{cris}}(\barrho_1, \barrho_2^{\square})[\WT(\barm)].$
And then we can use similar argument as in the $d=2$ case to conclude that both (S1) and (S2) are true.
\end{proof}

\subsection{Main local theorem}
\begin{thm} \label{main}
With notations in \textnormal{\textbf{(CRYS)}}, and suppose that the reduction $\overline \rho$ is upper triangular. Suppose that
\begin{itemize}
  \item Condition \textnormal{\textbf{(C-1)}} is satisfied, and
  \item Either \textnormal{\textbf{(C-2A)}} or \textnormal{\textbf{(C-2B)}} is satisfied.
\end{itemize}
Then there exists an upper triangular crystalline lift $\rho'$ of $\overline \rho$ such that $\HT_{s}(\rho)=\HT_{s}(\rho'), \forall s$.
\end{thm}

\begin{proof}
By \cite[Cor. 5.10]{GLS14}, we have
$\barhatm \in \mathcal E_{\vtshape}(\barhatn_d, \ldots, \barhatn_1).$
When \textnormal{\textbf{(C-2B)}} is satisfied, we can apply Theorem \ref{lifting-1}. When \textnormal{\textbf{(C-2A)}} is satisfied, we can equip $\barhatm$ with a successive $\Ext$ structure:
$\barhatm^{\square} \in
\mathcal E_{\vtshape}^{\square}(\barhatn_d, \ldots, \barhatn_1),$
and then apply Theorem \ref{lifting-2}.
\end{proof}

\begin{corollary}\label{maincorollary}
With notations in \textnormal{\textbf{(CRYS)}}, and suppose that the reduction $\overline \rho \in \mathcal E(\barchi_1, \ldots, \barchi_d)$ is upper triangular. Suppose one of the following conditions is satisfied:
\begin{enumerate}
\item $K=\Qp$, the differences between two elements in $\HT(D_0)$ are never $p-1$. And $\overline \chi_i^{-1} \overline \chi_j \neq \mathbbm{1}, \overline \varepsilon_p, \forall i<j.$

  \item For each $s$, the differences between two elements in $\HT(D_s)$ are never $1$. And for one $s_0$, $p-1 \notin \HT(D_{s_0})$.And $\overline \chi_i^{-1} \overline \chi_j \neq \mathbbm{1}, \overline \varepsilon_p, \forall i<j.$

  \item For each $s$, the differences between two elements in $\HT(D_s)$ are never $1$. For one $s_0$, $p-1 \notin \HT(D_{s_0})$. For one $0 \leq s_0' \leq f-1$, $p \notin \HT(D_{s_0'})$ (it is possible that $s_0 =s_0'$).

   \item For each $s$, $\HT(D_s) \subseteq [0, p-1]$. And for one $s_0$, $p-1 \notin \HT(D_{s_0})$.
  \end{enumerate}
Then there exists an upper triangular crystalline lift $\rho'$ of $\overline \rho$ such that $\HT_{s}(\rho)=\HT_{s}(\rho'), \forall s$.
\end{corollary}
\begin{proof}
\textnormal{\textbf{(C-1)}} is satisfied in all the 4 listed conditions. \textnormal{\textbf{(C-2A)}} is satisfied in (1) and (2). \textnormal{\textbf{(C-2B)}} is satisfied in (3) and (4).
\end{proof}

By using Fontaine-Laffaille theory, we can also prove some results along the line of our main theorem.
\begin{thm} [{ \cite[Lem. 1.4.2]{BLGGT14}, \cite[Thm. 3.0.3]{GL14}  }] \label{BLGGT}
Let $K/\Qp$ be a finite unramified extension (we can allow $p=2$ here), $\rho: G_K \to \GL_d(\mathcal O_E)$ a crystalline representation such that $\barrho$ is upper triangular. Suppose either of the following is true:
\begin{enumerate}
  \item $\HT_s(\rho) \subseteq [0, p-2]$ (not necessarily distinct) for all $s$.
  \item $\HT_s(\rho) \subseteq [0, p-1]$ (not necessarily distinct) for all $s$, and $\rho$ is unipotent.
\end{enumerate}
Then there exists an upper triangular crystalline representation $\rho'$ such that $\HT_s(\rho')=\HT_s(\rho)$ for all $s$ and $\overline \rho'\simeq \overline \rho$.
\end{thm}

\begin{rem}
It is clear that our Corollary \ref{maincorollary} can completely cover Case (1) in Theorem \ref{BLGGT}, but not Case (2). However, our result proves new cases that are not covered in Theorem \ref{BLGGT}. For example, in Corollary \ref{maincorollary}(2), it is possible we have $p \in \HT(D_{s_0''})$ for some $s_0''$ (even when $K =\Qp$). Also in Corollary \ref{maincorollary}(3), we do not need to assume that $\rho$ is unipotent (although we have restriction on the Hodge-Tate weights). More importantly, our methods can be used to prove similar results when $K$ is ramified, where Fontaine-Laffaille theory is not available, see forthcoming \cite{Gao15ram}.
\end{rem}

\section{Application to weight part of Serre's conjecture}

The local results proved in Section 7 have direct application to the weight part of Serre's conjecture for mod $p$ Galois representations associated to automorphic representations on unitary groups of rank $d$, as outlined in \cite{BLGG14}. We will only introduce necessary notations for our purposes, and the reader should refer to \cite{BLGG14} for any unfamiliar terms and more details. Note that our convention of Hodge-Tate weights are the opposite of \textit{loc. cit.}, and we use $p$ for their $\ell$, so we change the notations accordingly in our paper.

\textbf{(Notation-$F$):} Throughout this section, we suppose $p>2$. Let $F$ be an imaginary CM field, with maximal totally real subfield $F^{+}$, and denote $c \in \Gal(F/F^{+})$ the nontrivial element. Suppose that any place $v$ in $F^{+}$ over $p$ splits completely in $F$, that is, $v=\tilde v \tilde v^c$ in $F$ for a fixed choice of $\tilde v$. For each place $w\mid p$ of $F$, denote the completion as $F_w$, with residue field $k_w$. In the following, the subscript $w \mid p$ means all places in $F$ over $p$.

Let $\mathbb Z_{+}^d:= \{(a_1, \ldots, a_d) \in \mathbb Z^d \mid a_1 \geq \ldots  \geq a_d\}$. Define
$(\mathbb Z_{+}^d)_0^{\coprod_{w \mid p} \Hom(k_w, \Fpbar)}  \subset  (\mathbb Z_{+}^d)^{\coprod_{w \mid p} \Hom(k_w, \Fpbar)}$ a subset, where an element
  $a = \prod_{w \mid p} \prod_{\kappa \in \Hom(k_w, \Fpbar)} (a_{w, \kappa, 1}, \ldots, a_{w, \kappa, d})$
  is in the subset, if
$a_{w, \kappa, i}  +  a_{w, \kappa c, d+1-i} =0, \forall w \mid p, \kappa \in \Hom(k_w, \Fpbar), \textnormal{ and } \forall 1 \leq i \leq d.$

We call $a \in \mathbb Z_{+}^d$ a \emph{Serre weight} if $a_i-a_{i+1} \leq p-1, \forall i$. We call an element in a set of the shape $(\mathbb Z_{+}^d)^N$ a Serre weight, if each $\mathbb Z_{+}^d$-constituent is a Serre weight.

Now for any $K/\Qp$ unramified with residue field $k$, we can and do naturally identify $\Hom(K, \overline K)$ with $\Hom(k, \Fpbar)$ (i.e., identify $\kappa$ with its reduction $\overline \kappa$). Let $a \in (\mathbb Z_{+}^d)^{\Hom(K, \overline K)} =(\mathbb Z_{+}^d)^{\Hom(k, \Fpbar)}$, we call a de Rham representation $\rho: G_K \to \GL_d(\overline \Qp)$ of \emph{Hodge type} $a$, if
$\HT_{\kappa}(\rho) = \{a_{\kappa, 1}+d-1, a_{\kappa, 2}+d-2, \ldots, a_{\kappa, d} \}, \forall \kappa.$
For a residual representation $\barrho: G_K \to \GL_d(\Fpbar)$, let $W^{\textnormal{cris}}(\barrho)$ be the set of Serre weights $a$, such that $\barrho$ has a crystalline lift of Hodge type $a$. Also, let $W^{\textnormal{diag}}(\barrho)$ be the set of Serre weights $a$, such that $\barrho$ has a \emph{potentially diagonalizable} crystalline lift of Hodge type $a$. Here, potential diagonalizability is in the sense of \cite[\S 1.4]{BLGGT14}, and we omit the definition. But we remark that all upper triangular crystalline representations are potentially diagonalizable, which is an easy conclusion from the definition.
Now, we state our main result (see \cite{BLGG14} for any unfamiliar terms).

\begin{thm} \label{application}
With notations in the paragraph \textbf{(Notation-$F$)}. Suppose furthermore that $p$ is unramified in $F$.
Suppose $\overline r: G_F \to \GL_d(\Fpbar)$ is an irreducible representation with split ramification.
Assume that there is RACSDC automorphic representation $\Pi$ of $\GL_d(\mathbb{A}_F)$ of weight $\mu \in (\mathbb Z_{+}^{d})^{\Hom (F, \mathbb C)}$ and level prime to $p$ such that:
\begin{itemize}
  \item $\overline r \simeq \overline r_{p, \iota}(\Pi)$ (that is, $\overline r$ is automorphic).
  \item For each $\tau \in \Hom(F, \mathbb C)$, $\mu_{\tau, 1}-\mu_{\tau, d} \leq p-d$.
  \item $\overline r(G_{F(\zeta_p)})$ is adequate.
\end{itemize}

Suppose furthermore that for \emph{each} $w|p$, $\overline r \mid_{G_{F_w}} \in \mathcal E_{G_{F_w}}(\barchi_{w, 1}, \ldots, \barchi_{w, d})$ is upper triangular.
Now let
$a=(a_{w})_{w \mid p} \in (\mathbb Z_{+}^d)_0^{\coprod_{w \mid p}\Hom(k_w, \Fpbar)}$
be a Serre weight, such that
\begin{itemize}
  \item $a_{w, \kappa, 1}- a_{w, \kappa, d} \leq p-d+1, \forall w, \kappa$, and
  \item $a_w \in W^{\cris}(\overline r \mid_{G_{F_w}}), \forall w\mid p.$
\end{itemize}
And for \emph{each} $w\mid p$, \emph{any one} of the following listed 4 conditions is satisfied. Before we proceed to list the conditions, we make the following conventions on notations. Since we are fixing one $w$ each time, so for the brevity of notations, we can omit $w$ from all the subscripts. So we let $[k_w: \Fp]=f_w =f$, and write $a_w = a = \Pi_{s=0}^{f-1}(a_{s, 1}, \ldots, a_{s, d})$. Also simply write $\overline r \mid_{G_{F_w}} \in \mathcal E(\barchi_1, \ldots, \barchi_d)$.

\begin{enumerate}
\item $f=1$, i.e., $F_w =\Qp$. $a_{0, i}+(d-i)-a_{0, j}-(d-j) \neq p-1, \forall i<j$. And $\barchi_i^{-1} \barchi_j \neq \mathbbm{1}, \overline{\varepsilon}_p, \forall i<j$.
  \item $a_{s, i} \neq a_{s, j}, \forall s, \forall i\neq j$. For one $s_0$, $a_{s_0, i}+(d-i)-a_{s_0, j}-(d-j) \neq p-1, \forall i<j$. And $\barchi_i^{-1} \barchi_j \neq \mathbbm{1}, \overline{\varepsilon}_p, \forall i<j$.

  \item $a_{s, i} \neq a_{s, j}, \forall s, \forall i\neq j$. For one $s_0$, $a_{s_0, i}+(d-i)-a_{s_0, j}-(d-j) \neq p-1, \forall i<j$. And for one $s_0'$, $a_{s_0', 1}+(d-1)-a_{s_0', d} \neq p$ (it is possible that $s_0 =s_0'$).

  \item $a_{s, 1}+d-1-a_{s, d} \leq p-1, \forall s$. And for one $s_0$, $a_{s_0, 1}+d-1-a_{s_0, d} \leq p-2$.
\end{enumerate}

Then, $\overline r$ is automorphic of weight $a$. \footnote{In order to save space, we did not recall what it means for $\overline r$ to be automorphic of some Serre weight $a$. Roughly speaking, it means that there exists a degree $0$ cohomology class on some unitary group with coefficients in a certain local system corresponding to $a$, whose Hecke eigenvalues are determined by the characteristic polynomials of $\overline r$ at Frobenius elements. See \cite[Def. 2.1.9]{BLGG14}.}

\end{thm}

\begin{proof}
For each $w \mid p$, the listed 4 conditions on $a_w$ are precisely translated from those of Corollary \ref{maincorollary}. Since $a_w \in W^{\textnormal{cris}}(\overline r \mid_{G_{F_w}})$, so if any one of the listed 4 conditions is satisfied, then by Corollary \ref{maincorollary}, $\overline r \mid_{G_{F_w}}$ has a upper triangular crystalline lift with Hodge type $a_w$. That is to say, $a_w \in W^{\textnormal{diag}}(\overline r \mid_{G_{F_w}})$. And then we can apply \cite[Cor. 4.1.10]{BLGG14} to conclude. Note that in our case, $P_a=F_a$ (in the notation of \textit{loc. cit.}) is irreducible because our Serre weight lies in the closure of the lowest alcove, see e.g., \cite[Prop. 3.18]{Her09} for the $d=3$ case.
\end{proof}


\bibliographystyle{alpha}

\end{document}